\documentclass[11pt, leqno]{article}
\usepackage[margin=25mm]{geometry}
\usepackage{amsmath, amssymb, amsthm, amscd, mathrsfs}
\usepackage[all]{xy}
\usepackage{comment}
\usepackage{enumerate}

\numberwithin{equation}{subsection}

\theoremstyle{theorem}
\newtheorem{theorem}{Theorem}[subsection]
\newtheorem{definition}[theorem]{Definition}
\newtheorem{proposition}[theorem]{Proposition}
\newtheorem{lemma}[theorem]{Lemma}
\newtheorem{corollary}[theorem]{Corollary}
\newtheorem{conjecture}[theorem]{Conjecture}

\theoremstyle{definition}
\newtheorem{remark}[theorem]{Remark}{\rm}
{\rm}

\newcommand{\C}{\mathfrak{C}}
\newcommand{\D}{\mathcal{D}}
\newcommand{\DD}{\mathbb{D}}
\newcommand{\E}{\mathcal{E}}
\newcommand{\F}{\mathcal{F}}
\newcommand{\G}{\underline{\Gamma}^\dag}
\renewcommand{\H}{\mathcal{H}}
\renewcommand{\S}{\mathfrak{S}}
\renewcommand{\O}{\mathcal{O}}
\renewcommand{\P}{\mathfrak{P}}
\newcommand{\Q}{\mathbb{Q}}

\newcommand{\Rg}{\mathbb{R}\Gamma}
\newcommand{\RG}{\mathbb{R}\underline{\Gamma}^\dag}
\newcommand{\U}{\mathfrak{U}}
\newcommand{\V}{\mathcal{V}}

\newcommand{\X}{\mathfrak{X}}
\newcommand{\Y}{\mathfrak{Y}}
\newcommand{\Spec}{{\rm Spec}\,}
\newcommand{\Spf}{{\rm Spf}\,}
\newcommand{\card}{\sharp}
\newcommand{\coker}{{\rm coker}\,}
\newcommand{\im}{{\rm im}\,}
\newcommand{\one}{\mathbf{1}}

\makeatletter
\@addtoreset{equation}{section}
\makeatother

\title{$p$-Adic Weight Spectral Sequences of Strictly Semi-stable Schemes over Formal Power Series Rings via Arithmetic $\D$-modules}
\author{Yuanmin Liu}
\date{}
\begin{document}
\maketitle
    \begin{abstract}
        Let $k$ be a perfect field of characteristic $p > 0$. For a strictly semi-stable scheme over $k[[t]]$, we construct the weight spectral sequence in $p$-adic cohomology using the theory of arithmetic $\D$-modules, whose $E_1$ terms are described by rigid cohomologies of irreducible components of the closed fiber and whose $E_\infty$ terms are conjecturally described by the (unipotent) nearby cycle of Lazda-P\'{a}l's rigid cohomology over the bounded Robba ring. We also show its functoriality by pushforward and state the conjecture of its functoriality by pullback and dual.
    \end{abstract}
    
    \section{Introduction}
    For an algebraic variety $X$ smooth over $\mathbb{C}$, Deligne proved that the Betti cohomology $H^n(X_{\rm an}, \mathbb{Q})$ admits the mixed Hodge structure \cite{mHs}. In particular, he defined weight filtration on $H^n(X_{\rm an}, \Q)$ whose graded pieces are described by the cohomologies of proper smooth varieties. For a complex manifold proper over a complex disc with strictly semi-stable reduction at the origin, Steenbrink \cite{char-zero} constructed the mixed Hodge structure and the weight filtration on the Betti cohomology of the general fiber. Then it is natural to expect that we can define weight filtrations on other cohomology thories with some suitable properties. In the context of arithmetic geometry, Rapoport-Zink \cite{RZ} first obtained the weight spectral sequence in $l$-adic cohomology for proper strictly semi-stable scheme. Let $L/\Q_p$ be a finite field extension, $\O_L$ its valuation ring, and $F$ its residue field. Let $\bar{L}$ be an algebraic closure of $L$ and $\bar{F}$ that of $F$. For a proper strictly semi-stable scheme $X$ over $\O_L$, let $D_0, \dots, D_n$ be the irreducible components of the closed fiber of $X$ and put $\displaystyle D^{(p)} = \coprod_{i_0 < \cdots < i_p} D_{i_0} \cap \cdots \cap D_{i_p}.$ Then, for a prime $l \neq p$, the weight spectral sequence for $l$-adic cohomology is written as follows:
    $$E_1^{p,q} = \bigoplus_{i \geq \max(0, -p)} H^{q-2i}_{\rm et}(D^{(p+2i)}_{\bar{F}}, \Q_l(-i)) \Rightarrow H^{p+q}_{\rm et}(X_{\bar{L}}, \Q_l).$$
    Note that $E_1^{p,q}$ has weight $q$ by Weil conjecture. T. Saito \cite{l-adic} gave another definition of the weight spectral sequence for $l$-adic cohomology based on the perversity of nearby cycle sheaf up to shift and proved the functoriality of the weight spectral sequence. Also, Nakayama \cite{Naky} formulated the weight spectral sequence for $l$-adic cohomology only in terms of the log closed fiber of $X$ and the log $l$-adic cohomology.
    
    The weight spectral in $p$-adic cohomology was first obtained by Mokrane \cite{Mok1} \cite{Mok2} using log de Rham-Witt cohomology. There are also works by Nakkajima and Shiho \cite{Nak} \cite{NS} using log crystalline cohomology. However, in order to prove the functoriality of weight spectral sequences, it would be better to construct the $p$-adic weight spectral sequence using $p$-adic cohomology with good theory of coefficients over log schemes. In \cite{logD}, Tsuji took this point of view and constructed the $p$-adic weight spectral sequences using log $\D$-modules. However, the functoriality is not yet available in the literature and the theory of $\D$-modules he used is not the theory of (log) arithmetic $\D$-modules due to Berthelot, which is believed to be the good theory of coefficients in $p$-adic cohomology.
    
    The theory of log arithmetic $\D$-modules is not yet fully developed, although there is a work of Caro-Vauclair \cite{CV}. In this article, we suppose that we are given a projective strictly semi-stable scheme $X$ over $k[[t]]$ with $k$ a perfect field of characteristic $p > 0$ and construct the $p$-adic weight spectral sequence of it by using the theory of arithmetic $\D$-modules over $k[[t]]$ developed by Caro \cite{intro}. We follow the idea of T. Saito \cite{l-adic} and construct the weight spectral sequence by showing the holonomicity of the unipotent nearby cycle up to shift: The unipotent nearby cycle functor in arithmetic $\D$-modules is defined by Abe-Caro \cite{bei} over $k[t]$ and we can imitate their construction in the case over $k[[t]]$. The main theorem of this paper is described as follows:
    \\
    
    \noindent
    {\bf Theorem \ref{main-theorem}.}
        {\it We have a spectral sequence}
        $$ E^{p,q}_1 = \bigoplus_{i \geq \max(0, -p)} H^{q-2i}_{\D}(D^{(p+2i)}/K)(-i) \Rightarrow \Psi H^{p+q}_{\D}(X_\eta/\E^\dag_K).$$
    Here $X_\eta$ is the generic fiber of $X$, $D^{(p)}$ is defined as in the $l$-adic case, $H^n_{\D}$ denotes the arithmetic $\D$-module cohomology (see Definition \ref{ad-coh}) and $\Psi$ denotes the unipotent nearby cycle functor (see Definition \ref{un-def}).
    
    The arithmetic $\D$-module cohomology of varieties over $k$ is known to be isomorphic to the rigid cohomology \cite{comparison}, which describes $E_1$-terms. On the other hand, the arithmetic $\D$-module cohomology of varieties over $k((t))$ will be vector spaces over the bounded Robba ring $\E^\dag_K$ which is conjecturally isomorphic to the rigid cohomology over $\E^\dag_K$ of Lazda-P\'{a}l \cite{LPrig} (see Conjecture \ref{conj-comp}). Hence, under the conjecture, $E_\infty$-terms are described by the Lazda-P\'{a}l's rigid cohomology of $X_\eta$ over $\E^\dag_K$.
    
    Note that Lazda-P\'{a}l's rigid cohomology of $X_\eta$ over $\E^\dag_K$ is naturally a unipotent $\nabla$-module and its unipotent nearby cycle is isomorphic to the log crystalline cohomology of the log special fiber of $X$ tensored with $K$ (see \cite[Theorem 5.46]{LPrig}). We do not know the corresponding properties for the arithmetic $\D$-module cohomology of $X_\eta$. We expect that a suitable formalism of log arithmetic $\D$-modules will allow us to prove them.
    
    We will also show the functoriality of $p$-adic weight spectral sequences by pushforward. The functoriality by pullback follows from the compatibility with dual, which the author do not know how to prove (cf. Conjecture \ref{dual-nc} and Theorem \ref{dual-and-pull}).
    
    We explain the content of each section. In section \ref{sec-preliminaries}, we introduce the notion of schemes and formal schemes locally with finite $p$-basis. It is a slight generalization of smoothness: for example, $k[[t]]$ has finite $p$-basis over $k$, on which we can build the theory of arithmetic $\D$-modules as shown in \cite{intro}. Section \ref{sec-six} collects notions and propositions in order to formulate six functor formalism of arithmetic $\D$-modules and t-structures. Section \ref{sec-purity} is dedicated to prove the purity theorem \ref{purity} analogous to the $l$-adic case. In section \ref{sec-nc}, we define the unipotent nearby cycle functor $\Psi$ following the method of \cite{bei} and calculate $\Psi \one_{X_\eta}.$ We will use these results to define the $p$-adic weight spectral sequence in section \ref{sec-wss}. The functoriality will be discussed in section \ref{sec-func}.
    \newline
    
    \subsection*{Convention}
    Throughout this paper, we fix a perfect field $k$ of characteristic $p > 0$, a complete discrete valuation ring $\V$ of mixed characteristc $(0, p)$ whose residue field is $k$. We let $\varpi$ be a uniformizer of $\V$ and $K$ its quotient field. We assume that there exists an automorphism $\sigma: \V \rightarrow \V$ lifting some power of the Frobenius $F^s : k \rightarrow k$. We put $q := p^s$.
    
    Let $S = \Spec k, \, \S = \Spf \V, \, \DD_S = \Spec k[[t]], \, \DD_\S = \Spf \V[[t]].$ Here $\V$ and $\V[[t]]$ have $p$-adic topology. Let $s$ be the closed point and $\eta$ the generic point of $\DD_S$. All schemes we treat are of finte type over $\DD_S$ and all formal schemes are topologically of finite type over $\DD_\S$.
    
    We use fraktur characters such as $\P$ or $\X$ for formal schemes; we often denote without mentioning by $P$ and $X$ the schemes over $k$ obtained from the corresponding formal schemes modulo $\varpi$ (not $p$). For a morphism of formal schemes $f : \P \rightarrow \P'$, we denote by $f_0$ the induced morphism of schemes $f \mod \varpi : P \rightarrow P'.$
    
    \subsection*{Acknowledgement}
    The author would like to express appreciation for my supervisor Atsushi Shiho, who gave me helpful advices throughout writing this master thesis and for Daniel Caro, who showed interest to the subject of this thesis and answered my questions about nearby cycle.
    This work was supported by the World-leading INnovative Graduate Study Program for Frontiers of Mathematical Sciences and Physics, The University of Tokyo.
    
    \section{Preliminaries} \label{sec-preliminaries}
    \subsection{$p$-basis}
    Although $k[[t]]$ is not smooth over $k$ since it is not finitely generated as a $k$-algebra, it is formally smooth over $k$; more precisely, $k[t] \hookrightarrow k[[t]]$ is formally etale. Since we will treat schemes over $k[[t]]$, we need the notion of schemes ``locally with finite $p$-basis over $k$," which includes smooth schemes over $k$ and $k[[t]]$.
    
    \begin{definition} {\rm \cite[1.1]{rel-perf}}
        Let $f : X \rightarrow Y$ be a morphism of $k$-schemes. We say that $f$ is relatively perfect if the following diagram is cartesian:
        \[
        \xymatrix{
            X \ar[r]^{F_X} \ar[d]_{f} & X \ar[d]^{f}\\
            Y \ar[r]^{F_Y} & Y,\\
        }
        \]
        where $F_X$ and $F_Y$ are the absolute Frobenius.
    \end{definition}
    \begin{proposition}[{\cite[1.2]{rel-perf}}]
        Let $f : X \rightarrow Y$ be a morphism of $k$-schemes. Consider following properties:
        \\
        (i) $f$ is etale;
        \\
        (ii) $f$ is relatively perfect;
        \\
        (iii) $f$ is formally etale.
        \\
        Then we have $(i) \Rightarrow (ii) \Rightarrow (iii).$
    \end{proposition}
    
    \begin{definition}
        Let $f : X \rightarrow Y$ be a morphism of $k$-schemes. A (finite) $p$-basis of $X$ over $Y$ is finitely many elements $x_1, \dots, x_n \in \Gamma(X, \O_X)$ such that the corresponding morphism $X \rightarrow \mathbb{A}^n_Y$ is relatively perfect. In this situation $\Omega_{X/Y}$ is a free module generated by $dx_1, \dots, dx_n$. If $X \rightarrow Y$ locally has finite $p$-basis, we define the locally constant function $\delta_{X/Y} : X \rightarrow \mathbb N$ by $x \mapsto {\rm rank}_{\O_{X, x}} \Omega_{X/Y, x}$ (cf. {\rm \cite[3.2.7]{intro}}). When $Y = S$, we abbreviate it to $\delta_X$.
    \end{definition}
    
    \begin{definition}[{\cite[1.1.7]{intro}}]
        (i) Let $f : \mathfrak X \rightarrow \mathfrak Y$ be a morphism of formal $\S$-schemes. We say that $f$ is relatively perfect if $f$ is formally etale and $f_0 : X \rightarrow Y$ is relatively perfect.
        \\
        (ii) Let $f : \mathfrak X \rightarrow \mathfrak Y$ be a morphism of formal $\S$-schemes. A finite $p$-basis of $\mathfrak X$ over $\Y$ is finitely many elements $x_1, \dots, x_n \in \Gamma(\X, \O_\X)$ such that the corresponding morphism $\X \rightarrow \widehat{\mathbb{A}}^n_{\mathfrak Y}$ is relatively perfect.
    \end{definition}
    
    \subsection{Strictly semi-stable scheme over $k[[t]]$} \label{ssss}
        A strictly semi-stable scheme over $k[[t]]$ is an integral scheme $X$ of finite type over $k[[t]]$ such that the generic fiber $X_\eta$ is smooth over $\eta$ and the closed fiber $X_s$ is a strict normal crossings divisor of $X$. Note that since $k$ is perfect, this is equivalent to the usual definition \cite[1.2.1]{intro}.
        
        We let $D_0, \dots, D_r$ be the irreducible components of $X_s$. $X$ can be covered by open subsets $U$ etale over $\Spec k[[t]][t_0, \dots, t_d]/(t - t_0\cdots t_{r'})$ such that $D_i \cap U$ is defined by $t_i = 0$ for $i \leq r'$ and $D_i \cap U = \emptyset$ for $i > r'$ up to reordering. 
        Note that in this situation $t_0, \dots, t_d$ is a $p$-basis of $U$ and therefore $d+1 = \delta_X.$
    
    \section{Six functor formalism and t-structures} \label{sec-six}
    In this section, we establish the six functor formalism of holonomic arithmetic $\D$-modules of realizable schemes,  t-structures on whose triangulated categories and gluing theorem of ($F$-)overholonomic modules. We will frequently use notions and symbols on arithmetic $\D$-modules, following \cite{intro}.
    \subsection{Six functor formalism of frames and couples}
    \begin{definition}
        A frame over $\V[[t]]$ is a triple $(X, Y, \P)$ of a quasi-projective formal scheme $\P$ over $\V[[t]]$ locally with finite $p$-basis over $\V$, a closed subscheme $Y$ of $\P$ and an open subscheme $X$ of $Y$.
        A morphism of frames $\theta = (a, b, f) : (X', Y', \P') \rightarrow (X, Y, \P)$ is a commutative diagram
        \[
        \xymatrix{
            X' \ar@{^{(}->}[r] \ar[d]^a & Y' \ar@{^{(}->}[r] \ar[d]^b & \P' \ar[d]^f \\
            X \ar@{^{(}->}[r] & Y \ar@{^{(}->}[r] & \P. \\
        }
        \]
        
        A couple over $\V[[t]]$ is a pair $(X, Y)$ such that there is a frame $(X, Y, \P)$. We say that $(X, Y, \P)$ is an enclosing frame of $(X, Y).$
        A morphism of couples $(a, b) : (X', Y') \rightarrow (X, Y)$ is a commutative diagram
        \[
        \xymatrix{
            X' \ar@{^{(}->}[r] \ar[d]^a & Y' \ar[d]^b \\
            X \ar@{^{(}->}[r] & Y.
        }
        \]
        \noindent
        When $\theta = (a, b, f) : (X', Y', \P') \rightarrow (X, Y, \P)$ is a morphsim of frames, we say that $\theta$ is an enclosing morphism of $(a, b).$
        
        A realizable scheme is a $k[[t]]$-scheme $X$ such that there is a couple $(X, Y)$ with $Y$ proper over $k[[t]]$.
    \end{definition}
    
    \begin{definition}[{cf. \cite[13.1.4]{intro}}]
        A data of coefficients is a collection of strict full subcategories $\mathfrak{C}(\P)$ of $D^b_{\rm coh}(\D^\dag_{\P/\V, \Q})$ for each formal $\V[[t]]$-scheme $\P$ locally with finite $p$-basis over $\V$.
        
        For a data of coefficients $\mathfrak C$ and a frame $(X, Y, \P)$, we define the strict full subcategory $\mathfrak C(X, Y, \P)$ of $\mathfrak C(\P)$ to be the objects $\E$ such that there exists an isomophism of the form $\E \xrightarrow{\sim} \RG_X \E$.
    \end{definition}
    
    Since $\C(X, Y, \P)$ is independent of $Y$, we often denote it by $\C(X, \P)$.
    
    In the setting of arithmetic $\D$-modules, we can show the independence of the choice of enclosing frames or couples:
    
    \begin{theorem}[{\cite[14.2.2]{intro}}] \label{crys-phil}
        Let $\C$ be a data of coefficients such that for each formal scheme $\P$ over $\V[[t]]$ locally with finite $p$-basis over $\V$, $\mathfrak C(\P)$ contains $\O_{\P, \Q}$ and is stable under shifts, devissages, pushforwards, quasi-projective extraordinary pullbacks and local cohomological functors.
        
        Let $(id, b, f): (X, Y', \P') \rightarrow (X, Y, \P)$ be a morphism of frames such that $b$ is proper. Then $\RG_Xf^!$ and $f_+$ induce quasi-inverse equivalences between the categories $\C(X, Y, \P)$ and $\C(X, Y', \P').$
    \end{theorem}
    
    \begin{remark}
        In \cite{intro}, a frame $(X, Y, \P)$ over $\V[[t]]$ is defined to be a triple such that $\P$ is in addition smooth over $\V[[t]]$; however in the proof of {\cite[14.2.2]{intro}} he only uses the fact that $\P$ locally has finite $p$-basis over $\V$ and the theorem holds in our context as well.        
        By this theorem, for a couple $(X, Y)$, we define $\C(X, Y)$ to be $\C(X, Y, \P)$ for some enclosing frame $(X, Y, \P)$; for a realizable scheme $X$ over $k[[t]]$, we define $\C(X)$ to be $\C(X, Y)$ for some couple $(X, Y)$.
    \end{remark}
    
    \begin{definition}
        Let $\mathfrak C$ be a data of coefficients such that for each formal scheme $\P$ over $\V[[t]]$ locally with finite $p$-basis over $\V$, $\mathfrak C(\P)$ contains $\O_{\P, \Q}$ and is stable under shifts, devissages, pushforwards, quasi-projective extraordinary pullbacks and local cohomological functors.
        Let $\theta = (a, b, f) : (X', Y', \P') \rightarrow (X, Y, \P)$ be a morphism of frames.
        \\
        (i) Define the extraordinary pullback by $$\theta^! := \RG_{X'} f^! : \mathfrak C(X, Y, \P) \rightarrow \mathfrak C(X', Y', \P').$$
        (ii) If $b$ is proper, define the pushforward by $$\theta_+ := f_+ : \mathfrak C(X', Y', \P') \rightarrow \mathfrak C(X, Y, \P).$$
        (iii) If $\C$ is stable under duals, define the dual functor by $$\DD_{(X, Y, \P)} := \RG_X \DD_\P : \C(X, Y, \P)^\circ \rightarrow \C(X, Y, \P),$$
        and define the pullback by $$\theta^+ := \DD_{(X', Y', \P')}\,\theta^!\,\DD_{(X, Y, \P)}: \C(X, Y, \P) \rightarrow \C(X', Y', \P').$$
        (iv) If $b$ is proper and $\C$ is stable under duals, define the extraordinary pushforward by $$\theta_! = \DD_{(X, Y, \P)}\,\theta_+\,\DD_{(X', Y', \P')} : \C(X', Y', \P') \rightarrow \C(X, Y, \P).$$
        (v) If $\C$ is stable under tensor products, define $$- \tilde{\otimes}_{(X, Y, \P)} -  := - \otimes^\dag_{\O_{\P, \Q}} - [-\delta_P]: \C(X, Y, \P) \times \C(X, Y, \P) \rightarrow \C(X, Y, \P),$$
        where $\otimes^\dag_{\O_{\P, \Q}}$ is the usual weakly completed tensor products and define $$- \otimes - := \DD(\DD(-) \, \tilde{\otimes} \, \DD(-)).$$
    \end{definition}
    In \cite[13.2]{intro}, he constructs the data of coefficients $D^b_{\rm ovhol}$ of overholonomic complexes, which contains $\O$ and is local, stable under shifts, devissages, direct summands, quasi-projective extraordinary pullbacks, pushforwards, local cohomological functors, cohomologies, and duals. The heart of the canonical t-structure (cf. Definition \ref{can-t}) of $D^b_{\rm ovhol}(X, Y, \P)$ is denoted by ${\rm Ovhol}(X, Y, \P)$ and called the category of overholonomic modules. $F\mathchar`-{\rm Ovhol}(X, Y, \P)$ will denote the category of overholonomic modules with Frobenius structures. Note that when there is a divisor $T$ of $\P$ such that $Y-X = Y \cap T$, then ${\rm Ovhol}(X, Y, \P)$ is a full subcategory of ${\rm Mod}(\D^\dag_{\P, \Q})$.
    
    We let ${\rm Hol}_F(X, Y, \P)$ be the thick abelian full subcategory of ${\rm Ovhol}(X, Y, \P)$ generated by objects $\E$ which admits some $a$-th Frobenius structure $F^{*a}\E \rightarrow \E.$ $D^b_{\rm hol}(X, Y, \P)$ is defined to be the full subcategory of $D^b_{\rm ovhol}(X, Y, \P)$ consisting of objects $\E$ such that $\H^i(\E) \in {\rm Hol}_F(X, Y, \P)$. This is the analogue of [AC14, 1.7]. $D^b_{\rm hol}$ is supposed to be stable under tensor products in addition to stability properties of $D^b_{\rm ovhol}.$ However, the stability relies on Kedlaya's semistable reduction theorem, which is necessary for proving that overconvergent $F$-isrocrystals are overholonomic; this is unknown for varieties over $k[[t]]$. In this article, each time when $\E \otimes \E'$ appears we can always show that it belongs to $D^b_{\rm hol}$.
    \newline
    
    We end this subsection by stating a gluing theorem of ($F$-)overholonomic $\D$-modules. Let $(X, Y, \P)$ be a frame and suppose that there is a divisor $T$ of $\P$ such that $Y - X = Y \cap T.$ Choose an affine open covering $(\P_\alpha)_{\alpha \in A}$ of $\P$. For $\alpha_1, \dots, \alpha_n \in A$, put $\P_{\alpha_1 \cdots \alpha_n} = \P_{\alpha_1} \cap \cdots \cap \P_{\alpha_n}, \ Y_{\alpha_1 \cdots \alpha_n} = Y \cap \P_{\alpha_1 \cdots \alpha_n}$ and $X_{\alpha_1 \cdots \alpha_n} = X \cap \P_{\alpha_1 \cdots \alpha_n}$.

    We suppose that we have closed immersions $u_{\alpha_1} : \Y_{\alpha_1} \rightarrow \P_{\alpha_1}, \ u_{\alpha_1\alpha_2} : \Y_{\alpha_1\alpha_2} \rightarrow \P_{\alpha_1\alpha_2}$ and $u_{\alpha_1\alpha_2\alpha_3} : \Y_{\alpha_1\alpha_2\alpha_3} \rightarrow \P_{\alpha_1\alpha_2\alpha_3}$ of formal schemes locally with finite $p$-basis over $\V$ which lifts the closed immersion
    $Y_{\alpha_1} \rightarrow P_{\alpha_1}, \ Y_{\alpha_1\alpha_2} \rightarrow P_{\alpha_1\alpha_2}$ and $Y_{\alpha_1\alpha_2\alpha_3} \rightarrow P_{\alpha_1\alpha_2\alpha_3}$, respectively, for all $\alpha_1, \alpha_2, \alpha_3 \in A$. Furthermore suppose that we have flat lifts $p^i_{\alpha_1\alpha_2} : \Y_{\alpha_1\alpha_2} \rightarrow \Y_{\alpha_i}$ for $i = 1, 2$, $p^{ij}_{\alpha_1\alpha_2\alpha_3} : \Y_{\alpha_1\alpha_2\alpha_3} \rightarrow \Y_{\alpha_i\alpha_j}$ for $1 \leq i < j \leq 3$ and $p^{i}_{\alpha_1\alpha_2\alpha_3} : \Y_{\alpha_1\alpha_2\alpha_3} \rightarrow \Y_{\alpha_i}$ for $i = 1,2,3.$
    
    Note that such collection of lifts $u$ and $p$ exist if $\P$ is separated, as $\P_{\alpha_1 \cdots \alpha_n}$ are affine.
    
    \begin{definition}
        We call $((\P_\alpha), (\Y_\alpha), u)$ a data of lifts of the frame $(X, Y, \P).$
    \end{definition}
    
    \begin{definition}
        The category ${\rm (}F{\rm )}\mathchar`-{\rm Ovhol}((\Y_\alpha)_\alpha, T)$ is defined as follows:
        
        An object is a collection of data $(\E_\alpha, \theta_{\alpha\beta})$ where $\E_\alpha \in {\rm (}F{\rm )}\mathchar`-{\rm Ovhol}(X_\alpha, Y_\alpha, \Y_\alpha)$ and $\theta_{\alpha\beta} : p^{2!}_{\alpha\beta}\E_\beta \xrightarrow{\sim} p^{1!}_{\alpha\beta} \E_\alpha$ is an isomorphism such that for any $\alpha_1, \alpha_2, \alpha_3,$ $\theta_{\alpha_1\alpha_2\alpha_3}^{13} = \theta_{\alpha_1\alpha_2\alpha_3}^{12} \theta_{\alpha_1\alpha_2\alpha_3}^{23}$ where $\theta_{\alpha_1\alpha_2\alpha_3}^{ij}$ is the isomorphism which fits into the following commutative diagram:
        \[\xymatrixcolsep{20mm}
        \xymatrix{
            p^{ij!}_{\alpha_1\alpha_2\alpha_3}p^{2!}_{\alpha_i\alpha_j} \E_{\alpha_j} \ar[r]^{p^{ij!}_{\alpha_1\alpha_2\alpha_3} \theta_{\alpha_i\alpha_j}} \ar[d]^{\wr} & p^{ij!}_{\alpha_1\alpha_2\alpha_3}p^{1!}_{\alpha_i\alpha_j} \E_{\alpha_i} \ar[d]_{\wr}\\
            (p^{2}_{\alpha_i\alpha_j} p^{ij}_{\alpha_1\alpha_2\alpha_3})^!\E_{\alpha_j} \ar[d]^{\wr}_{\tau} & (p^{1}_{\alpha_i\alpha_j} p^{ij}_{\alpha_1\alpha_2\alpha_3})^!\E_{\alpha_i} \ar[d]^{\tau}_{\wr}\\
            p^{j!}_{\alpha_1\alpha_2\alpha_3}\E_{\alpha_j} \ar[r]^{\theta_{\alpha_1\alpha_2\alpha_3}^{ij}} & p^{i!}_{\alpha_1\alpha_2\alpha_3}\E_{\alpha_i};           
        }
        \]
        here $\tau := \tau_{f, f'}$ is the canonical isomorphism $f^! \xrightarrow{\sim} f'^!$ defined for morphisms $f, f' : \X \rightarrow \Y$ of formal schemes locally with finite $p$-basis over $\V$ such that $f_0 = f'_0$. For more details, see {\rm [Car21, 8.5.1]}.
        
        A morphism $\varphi : (\E_\alpha, \theta_{\alpha\beta}) \rightarrow (\E'_\alpha, \theta'_{\alpha\beta})$ is a collection of morphisms $\varphi_\alpha : \E_\alpha \rightarrow \E'_\alpha$ compatible with $\theta_{\alpha\beta}$ and $\theta'_{\alpha\beta}$.
    \end{definition}
    \noindent
    
    We can define the canonical functor $u_0^* : {\rm (}F{\rm )}\mathchar`-{\rm Ovhol}(X, Y, \P) \rightarrow {\rm (}F{\rm )}\mathchar`-{\rm Ovhol}((\Y_\alpha)_\alpha, T)$ by 
    $\E \mapsto (u_\alpha^!(\E|_{\P_\alpha}), c_{\alpha\beta})$ where $c_{\alpha\beta}$ is defined by $\tau$. Note that by Berthelot-Kashiwara theorem, $u_\alpha^!(\E|_{\P_\alpha})$ is an overholonomic $\D$-module.

    \begin{theorem} \label{u_0^*}
         $u_0^*$ is an equivalence of categories.
    \end{theorem}
    \begin{proof}
        We can prove in the same way as \cite[5.3.8]{intro0} that there is the equivalence of categories $u_0^*: {\rm  Coh}(X, Y, \P) \rightarrow {\rm Coh}((\Y_\alpha), T)$ with ${\rm Ovhol}$ replaced by ${\rm Coh}$, the category of coherent $\D$-modules. Note that the data of coefficients $D^b_{\rm ovhol}$ is local, i.e., for any $\E \in D^b_{\rm coh}(\P)$ an open covering $\{\U_i\}_i$ of $\P$, if any $\E|\U_i$ is overholonomic, then so is $\E$. Hence $u_0^*$ induces the equivalence ${\rm Ovhol}(X, Y, \P) \rightarrow {\rm Ovhol}((\Y_\alpha), T)$. Since the Frobenius $F^*$ is compatible with $\tau$ by \cite[2.2.6]{BerFrob}, the functor $u_0^*$ is also compatible with the Frobenius.
    \end{proof}
    
    We can show the compatibility of $u_0^*$ and the dual functor: define the dual functor $\DD$ on ${\rm (}F{\rm )}\mathchar`-{\rm Ovhol}((\Y_\alpha)_\alpha, T)$ by $\DD(\E_\alpha, \theta_{\alpha\beta}) := (\DD_{(X_\alpha, Y_\alpha, \Y_\alpha)}(\E_\alpha), \theta'_{\alpha\beta})$ where $\theta'_{\alpha\beta}$ is the composition $p^{2!}_{\alpha\beta}\DD(\E_\beta) \simeq \DD p^{2!}_{\alpha\beta}(\E_\beta) \xrightarrow{\DD \theta_{\alpha\beta}^{-1}} \DD p^{1!}_{\alpha\beta}(\E_\alpha) \simeq p^{1!}_{\alpha\beta}\DD(\E_\alpha)$.
    
    \begin{proposition} \label{u_0^*-dual}
        $u_0^*\DD \simeq \DD u_0^*.$
    \end{proposition}
    \begin{proof}
        Since $\E$ is supported on $X$, it follows from the commutativity of extraordinary pullback functors and dual functors.
    \end{proof}
    
    We can also show the compatibility of $u_0^*$ and the tensor product. 
    \begin{proposition} \label{u_0^*-tensor}
        Let $\E, \ \F \in {\rm (}F{\rm )}\mathchar`-{\rm Ovhol}(X, Y, \P)$ and suppose that $\E \tilde{\otimes}_{(X, Y, \P)} \F \in {\rm (}F{\rm )}\mathchar`-{\rm Ovhol}(X, Y, \P).$ Let $u_0^*\E = (\E_\alpha, \theta_{\alpha\beta}), \, u_0^* \F = (\F_\alpha, \eta_{\alpha\beta})$. Then
        $$u_0^*(\E \tilde{\otimes}_{(X, Y, \P)} \F) \simeq (\E_\alpha \tilde{\otimes}_{(X_\alpha, Y_\alpha, \Y_\alpha)} \F_\alpha, \theta_{\alpha\beta} \tilde{\otimes}_{(X_{\alpha\beta}, Y_{\alpha\beta}, \Y_{\alpha\beta})} \eta_{\alpha\beta})$$
    \end{proposition}
    \begin{proof}
        Use the commutativity of extraordinary pullback and the tensor products \cite[7.2.3]{intro}.
    \end{proof}
    
    \subsection{Constructible t-structure by devissage}
    Let $(X, Y, \P)$ be a frame and let $\C$ be a data of coefficients which is local, stable under shifts, devissages, direct summands, quasi-projective extraordinary pullbacks, pushforwards, local cohomological functors, cohomologies, and duals.
    
    \begin{definition} \label{can-t}
        The canonical (holonomic) t-structure of $\C(X, Y, \P)$ is defined as follows:
        
        for $\E \in \C(X, Y, \P)$, we define $\E \in \C^{\lesseqqgtr 0}(X, Y, \P)$ if and only if for some (any) open subset $\mathcal U \subseteq \P$ which contains $X$ as a closed subset, $\E|\mathcal U \in D^{\lesseqqgtr 0}_{\rm coh}(\D^\dag_{\U, \Q})$ 
    \end{definition}
    
    Before defining constructible t-structures, we need some preparations.
    
    \begin{definition}
        Suppose that $X$ locally has finite $p$-basis over $k$. We say that $\E \in \C^{\heartsuit}(X, \P)$ is an isocrystal if there exist an open subset $\U \subseteq \P$ which contains $X$ as a closed subset, an open covering $\{\P_\alpha\}$ of $\U$ and closed immersions $u_\alpha : \X_\alpha \hookrightarrow \P_\alpha$  which lift the closed immersions $X_\alpha := X \cap P_\alpha \hookrightarrow P_\alpha$ such that $u_\alpha^!(\E|\P_\alpha)$ is coherent over $\O_{\X, \Q}.$
        $\C_{\rm isoc}(X, \P)$ is the full subcategory of $\C(X, \P)$ consisting of objects $\E$ such that $\H^i(\E)$ are isocrystals for all $i$.
        We define $\C^{\lesseqqgtr 0}_{\rm isoc} := \C^{\lesseqqgtr 0} \cap \C_{\rm isoc}$.
    \end{definition}
    
    \begin{definition}
        A $p$-smooth stratification of $X$ is a descending chain of (reduced) closed subschemes $X_{\rm red} = X_0 \supseteq X_1 \supseteq \cdots \supseteq X_m = \emptyset$ such that $X_i - X_{i+1}$ locally have finite $p$-basis over $k$.
    \end{definition}
    
    \begin{definition}[{\cite[14.4]{intro}}] \label{def-cons}
        For $n \geq 0$, we define $\P^{(n)} = \P \times_{\Spf \V[[t]]} \Spf \V[[t^{1/p^n}]]$. Note that the projection $p^{(n)} : \P^{(n)} \rightarrow \P$ is relatively perfect since so is $\Spf \V[[t^{1/p^n}]] \rightarrow \Spf \V[[t]]$. For $\E \in \C(\P)$, we put $\E^{(n)} = p^{(n)+} \E \simeq p^{(n)!} \E.$
        For a locally closed subscheme $X' \subseteq X$, we let $\iota_{X'} : (X', \overline{X'}, \P) \rightarrow (X, Y, \P)$ be the natural locally closed immersion of frames, where $\overline{X'}$ is the closure of $X'$ in $\P$.
        The constructible t-sturcture of $\C(X, Y, \P)$ is defined as follows:
        $\E \in \C^{\leq 0}_{\rm cons}(X, Y, \P)$ (resp. $\E \in \C^{\geq 0}_{\rm cons}(X, Y, \P)$) if and only if there exist an $n \geq 0$ and a $p$-smooth stratification $X_0 \supseteq X_1 \supseteq \cdots \supseteq X_m = \emptyset$ of $X^{(n)} := (X \times_{\Spec k[[t]]} \Spec k[[t^{1/p^n}]])_{\rm red}$ such that 
        $\iota_{U_i}^+(\E^{(n)})[\delta_{U_i}] \in \C_{\rm isoc}^{\leq 0}(U_i, \P^{(n)})$ (resp. $\iota_{U_i}^!(\E^{(n)})[\delta_{U_i}] \in \C_{\rm isoc}^{\geq 0}(U_i, \P^{(n)})$),
        where $U_i = X_i - X_{i+1}$.
    \end{definition}
    
    \begin{remark}
        Note that $p^{(n)} : \P^{(n)} \rightarrow \P$ is finite locally free, surjective and radicial; therefore universally homeomorphic. By \cite[11.4.8]{intro}, $p^{(n)!}$ and $p^{(n)}_+$ induce quasi-inverse equivalences between the categories $\C(\P)$ and $\C(\P^{(n)})$.
    \end{remark}
    
    \begin{proposition}[Extension property] \label{ext-prop}
        Let $\E' \rightarrow \E \rightarrow \E'' \rightarrow$ be a distinguished triangle of $\C(X, Y, \P).$ If $\E', \E'' \in \C^{\lesseqqgtr 0}_{\rm cons}(X, Y, \P)$, then so is $\E$.
    \end{proposition}
    \begin{proof}
        Take $n', n'' \geq 0$ and $p$-smooth stratifications $(X'_i)$ of $X^{(n')}$ and $X^{(n'')}$ such that the condition of Definition \ref{def-cons} hold for $\E'$ with respect to $(X'_i)$ and for $\E''$ with respect to $(X''_i)$. Let $n = \max(n', n'')$ and take a $p$-smooth stratification $(X_i)$ of $X^{(n)}$ which refines the pullbacks of $(X'_i)$ and $(X''_i)$. Then the condition of Definition \ref{def-cons} hold for both $\E'$ and $\E''$ with respect to $(X_i)$. Then we see that the condition holds for $\E$ with respect to $(X_i)$.
    \end{proof}
    
    Let us prove that the constructible t-structure is indeed a t-structure.
    
    \begin{lemma} \label{recol_oc}
        Let $j : U \hookrightarrow X$ be an open immersion and $i : Z \hookrightarrow X$ its complementary closed immersion of realizable schemes. The functors 
        \begin{align*}
            j_!, \, j_+ : \C(U) \rightarrow \C(X),\\
            j^!, \, j^+ : \C(X) \rightarrow \C(U),\\
            i_!, \, i_+ : \C(Z) \rightarrow \C(X),\\
            i^!, \, i^+ : \C(X) \rightarrow \C(Z)
        \end{align*}
        satisfy the following properties:
        
        (i) $i_! \simeq i_+$ and $j^! \simeq j^+$;
        
        (ii) $i^!i_+ \simeq i^+i_+ \simeq id$, $j^+j_! \simeq j^+j_+ \simeq id$, $i^+j_! \simeq i^!j_+ \simeq 0$, and $j^+i_+ \simeq 0$;
        
        (iii) $(i^+, i_+)$, $(j_!, j^+)$, $(i_+, i^!)$ and $(j^+, j_+)$ are adjoint pairs;
        
        (iv) we have distinguished triangles $j_!j^+ \rightarrow id \rightarrow i_+i^+ \rightarrow$ and $i_+i^! \rightarrow id \rightarrow j_+j^+ \rightarrow $;
        
        (v) ({\rm \cite[1.3.2]{weight}}) $i^!, \, j^!$ and are left t-exact and $i^+, j^+$ are right t-exact.
    \end{lemma}
    \begin{proof}
        Choose a frame $(X, Y, \P)$. Then the assertions follows from the definition.
    \end{proof}
    
    \begin{lemma}[{\cite[1.4.10]{perv}}] \label{glue_t}
        Let $j : U \rightarrow X$ be an open immersion and $i : Z \rightarrow X$ its complementary closed immersion. Suppose that we have triangulated categories $D(U), D(Z), D(X)$ and six functor formalism as Lemma \ref{recol_oc} (i)-(iv). Suppose furthermore that we have full triangulated subcategories $T(U)$ of $D(U)$ and $T(Z)$ of $D(Z)$ with t-structures such that $i^+j_+$ maps $T(U)$ into $T(Z)$.
        Then the full subcategory $T(X, U) := \{ E \in D(X) : j^+E \in T(U), \,i^+E, i^!E \in T(Z) \}$ is a triangluated subcategory of $D(X)$.
        If we define
        
        $T^{\leq 0}(X, U) := \{ E \in T(X, U) : j^+E \in T^{\leq 0}(U), \,i^+E \in T^{\leq 0}(Z) \}$ and
        
        $T^{\geq 0}(X, U) := \{ E \in T(X, U) : j^+E \in T^{\geq 0}(U), \,i^!E \in T^{\geq 0}(Z) \}$,
        \\
        then this is a t-structure on $T(X, U).$
    \end{lemma}
    
    \begin{lemma} \label{cons_isoc}
        Suppose that $X$ locally has finite $p$-basis over $k$ and $\E \in \C_{\rm isoc}(X, \P).$ Then $\E \in \C^{\lesseqqgtr 0}_{\rm cons}(X, \P)$ if and only if $\E[\delta_X] \in \C^{\lesseqqgtr 0}(X, \P).$
    \end{lemma}
    \begin{proof}
        'If' part is clear from the definition. Suppose $\E \in \C^{\leq 0}_{\rm cons}(X, \P).$ By base change, there is an open immersion $j : U \hookrightarrow X$ such that $j^+\E \in \C^{\leq 0}_{\rm cons}(U, \P)$ and $i^+\E[\delta_Z] \in \C^{\leq 0}_{\rm isoc}(Z, \P)$, where $i : Z \hookrightarrow X$ is the complement of $j$: here $Z$ is chosen to be the last (nonempty) closed subscheme of some $p$-smooth stratification of $X$. By induction on the number of strata, we may assume that $j^+\E[\delta_U] \in \C^{\leq 0}(U, \P).$
        
        By shrinking $\P$, we may assume that $X$ is closed in $\P$. Since $j^+\E \in \C^{\leq \delta_U}(U, \P)$, $j^+$ is right t-exact and $\delta_U = \delta_X$, the truncation (for the canonical t-structure) $\tau^{\geq \delta_X+1}\E$ is supported on $Z$. Consider the following distinguished triangle:
        $$i^+\tau^{\leq \delta_X}\E \rightarrow i^+\E \rightarrow i^+\tau^{\geq \delta_X + 1}\E \rightarrow.$$
        Since $i^+$ is right t-exact by Lemma \ref{recol_oc}, $i^+\tau^{\leq \delta_X}\E \in \C^{\leq \delta_X}(Z, \P)$. By assumption we have $i^+\E \in \C^{\leq \delta_Z}_{\rm isoc}(Z, \P).$ Therefore we have $i^+\tau^{\geq \delta_X + 1}\E \in \C^{\leq \delta_X}(Z, \P).$
        Then, since $i_+i^+\tau^{\geq \delta_X + 1}\E \simeq \tau^{\geq \delta_X + 1}\E$, we have $\tau^{\geq \delta_X + 1}\E = 0$.
        
        The case for $\geq 0$ is proved by duality.
    \end{proof}
    
    \begin{lemma}[{\cite[14.1.11]{intro}}] \label{dev_isoc}
        Let $\E \in \C(X, Y, \P)$. Then there is an $n \geq 0$ and a $p$-smooth stratification $X_0 \supseteq X_1 \supseteq \cdots \supseteq X_m = \emptyset$ of $(X \times_{\Spec k[[t]]} \Spec k[[t^{1/p^n}]])_{\rm red}$ such that $\RG_{U_i} \E^{(n)} \in \C_{\rm isoc}(U_i, \P^{(n)}),$ where $U_i = X_i - X_{i+1}$.
    \end{lemma}
    
    \begin{theorem}
        $\C^{\lesseqqgtr 0}_{\rm cons}$ is a t-structure.
    \end{theorem}
    \begin{proof}
        Let $U \subseteq X$ be a nonempty open subset and let $\C'(U)$ be the full (triangulated) subcategory of $\C(U, \P)$ consisting of objects $\E$ such that $U^{(n)} := (U \times_{\Spec k[[t]]} \Spec k[[t^{1/p^n}]])_{\rm red}$ locally has finite $p$-basis over $k$ and $\E^{(n)} \in \C_{\rm isoc}(U^{(n)}, \P^{(n)})$ for some $n \geq 0$.
        
        By Lemma \ref{cons_isoc}, $\C'^{\lesseqqgtr 0}(U) := \C'(U) \cap \C^{\lesseqqgtr 0}_{\rm cons}(U, \P) = \C'(U) \cap \C^{\lesseqqgtr \delta_U}(U, \P)$ is a t-structure on $\C'(U)$.
        Suppose by noetherian induction that the constructible t-structure on any closed subset $Z \subsetneq X$ is indeed a t-structure. By the gluing lemma \ref{glue_t}, the full subcategory $T(X, U) := \{\E \in \C(X, \P) : j^+\E \in \C'(U)\}$ is a triangluated subcategory with a t-structure $T^{\lesseqqgtr 0}(X, U) = \C^{\lesseqqgtr 0}_{\rm cons}(X, \P) \cap T(X, U)$.
        Finally, Lemma \ref{dev_isoc} implies that by varying $U$, $T(X, U)$ covers $\C(X, \P)$; it is now straightforward to prove the statement.
    \end{proof}
    
    \subsection{Another definition of constructible t-structure}
    In this section, we let $\C = D^b_{\rm hol}.$ Abe defines in \cite[1.3]{lang} the constructible t-structure as follows:
    \begin{itemize}
        \item $\C^{\geq 0}_{\rm c}(X)$ consists of objects $K$ such that $\dim({\rm Supp}\, \H^n(K)) \leq n$ for all $n \geq 0$ and $\H^n(K) = 0$ for $n < 0$;
        \item $\C^{\leq 0}_{\rm c}(X)$ consists of objects $K$ such that $\H^n i^+_W = 0$ for any closed immersion $i : W \hookrightarrow X$ and $n > \dim(W).$
    \end{itemize}
    Here ${\rm Supp} \,\E =: Z$ is the smallest closed subset of $X$ such that $\E|_{X-Z} = 0.$
    
    We can prove that $\C^{\lesseqqgtr 0}_{\rm c}$ is a t-structure by slightly modifying the proof of \cite[1.3.3]{lang}; using the symbols of {\it loc. cit.}, $\bigcup_U T(X, U) = D(X)$ is not true in our context since $X$ might not have any nonempty open subset $U$ locally with finite $p$-basis over $k$. We have to define $T(X, U)$ to be the subcategory consisting of objects $\E \in D(X)$ such that $U^{(n)} := U \times_{\Spec k[[t]]} \Spec k[[t^{1/p^n}]]$ locally has finite $p$-basis over $k$ and $j^{(n)+} \E^{(n)} \in D^b_{\rm hol, isoc}(U^{(n)}).$
    
    By the same slight modification of \cite[2.4.3]{cunip}, we can prove that two definitions of constructible t-structures are equivalent.
    
    We have the following theorem also in our setting by the same proof:
    
    \begin{theorem}[{\cite[1.3.4]{lang}}] \label{ct-exact-pullback}
        For any morphism $f$, $f^+$ is t-exact for the constructible $t$-structure.
        For any locally closed immersion $i$, $i_!$ is t-exact for the constructible $t$-structure.
    \end{theorem}
    
    Lastly we remark that the equivalence of categories $u_0^* : F\mathchar`-{\rm Ovhol}(X, Y, \P) \rightarrow F\mathchar`-{\rm Ovhol}((\Y_\alpha)_\alpha, T)$ preserves constructible t-structures.
    
    \begin{proposition} \label{u_0^*-cons-t}
        In the situation of Theorem \ref{u_0^*}, $\E \in F\mathchar`-{\rm Ovhol}(X, Y, \P)$ is in $D^{b, \lesseqqgtr 0}_{\rm hol, cons}(X, Y, \P)$ if and only if $u_\alpha^!(\E|\P_\alpha)$ is in $D^{b, \lesseqqgtr 0}_{\rm hol, cons}(X_\alpha, Y_\alpha, \Y_\alpha)$ for all $\alpha$.
    \end{proposition}
    \begin{proof}
        If $\E \in F\mathchar`-{\rm Ovhol}(X, Y, \P)$, then $u^+ \E \simeq u^! \E.$ As constructible t-structure is Zariski local, Theorem \ref{ct-exact-pullback} implies the assertion.
    \end{proof}
    
    \subsection{Dual constructible t-structure}
    The dual constructible t-structure is defined by the dual of constructible t-structure:

    \begin{itemize}
        \item $\E \in \C^{\geq 0}_{\rm dcons}$ if and only if $\DD \E \in \C^{\leq 0}_{\rm cons}$;
        \item $\E \in \C^{\leq 0}_{\rm dcons}$ if and only if $\DD \E \in \C^{\geq 0}_{\rm cons}$.
    \end{itemize}
    
    By duality and results in the previous section, we have following properties:
    
    \begin{itemize}
        \item for any morphism $f$, $f^!$ is t-exact for the dual constructible t-structure;
        \item for any locally closed immersion $i$, $i_+$ is t-exact for the dual constructible t-structure;
    \end{itemize}
    
    Let $j$ be an open immersion and $i$ its complement. Since we have a distinguished triangle $$i_+i^! \rightarrow 1 \rightarrow j_+j^! \rightarrow$$
    of t-exact functors, by the 9-lemma, we get the following lemma:
    
    \begin{lemma} \label{dev_exact}
        Let $\Sigma : 0 \rightarrow \E' \xrightarrow{f} \E \xrightarrow{g} \E'' \rightarrow 0$ be a sequence of $D^{b,  \heartsuit}_{\rm hol, dcons}$ such that $gf = 0$. Then $\Sigma$ is an exact sequence if and only if $i^!\Sigma$ and $j^!\Sigma$ are exact sequences.
    \end{lemma}
    
    \section{Purity} \label{sec-purity}
    \subsection{Calculation of $\RG_Y(\O_{\P, \Q})$} \label{RG-calc}
    Let $\P$ be a projective formal scheme over $\V[[t]]$ locally with finite $p$-basis over $\V$, $Z_0, \dots, Z_n$ closed subschemes of $P$ of codimension $r$ and let $Y = Z_0 \cup \cdots \cup Z_n$. For $0 \leq i_0 < \cdots < i_p \leq n$, we put $Z_{i_0 \cdots i_p} = Z_{i_0} \cap \cdots \cap Z_{i_p}$. We suppose that locally there exists a finite $p$-basis $t_i$ of $\P$ over $\S$ such that for any $i_0 < \cdots < i_p$, $Z_{i_0\cdots i_p}$ is defined by $t_{j_0 } = \cdots = t_{j_{p+r}} = 0$ for some distinct $j_0 < \cdots < j_{p+r}$. For each $i = 0, \dots, r$, we fix a set of divisors $T_{i, 1}, \dots, T_{i, r}$ of $P$ such that $Z_i = T_{i, 1} \cap \cdots \cap T_{i, r}.$
    
    \begin{lemma}[{\cite[2.5.3]{intro}}] \label{sp_resol}
        Let $T_0, \dots, T_m$ be divisors of $P$ and put $U_i = P - T_i$. We define the complex $\mathcal{C}^\dag(T_0, \dots, T_m)$ to be
        $$\displaystyle \O_{\P^{\rm ad}} \rightarrow \bigoplus_{i_0} j^\dag_{U_{i_0}}\O_{\P^{\rm ad}} \rightarrow \bigoplus_{i_0<i_1} j^\dag_{U_{i_0} \cap U_{i_1}}\O_{\P^{\rm ad}} \rightarrow \cdots \rightarrow j^\dag_{U_0 \cap \cdots \cap U_m}\O_{\P^{\rm ad}},$$
        where the left-most term is put at degree 0. Then there is a canonical quasi-isomorphism $\G_{T_0 \cap \cdots \cap T_m} \O_{\P^{\rm ad}} \xrightarrow{\sim} \mathcal{C}^\dag(T_0, \dots, T_m).$ Furthermore, $\mathcal C^\dag(T_0, \dots, T_m)$ is an ${\rm sp}_*$-acyclic resolution of $\G_{T_0 \cap \cdots \cap T_m}\O_{\P^{\rm ad}}$.
        
        For $0 \leq i_0 < \cdots < i_{m'} \leq m$, the quasi-isomorphism is compatible with the canonical map of complexes $\mathcal C^\dag(T_0, \dots, T_m) \rightarrow \mathcal C^\dag(T_{i_0}, \dots, T_{i_{m'}})$.
    \end{lemma}
    
    By this lemma, we calculate $\RG_{T_0 \cap \cdots \cap T_m} \O_{\P, \Q} = \mathbb{R} {\rm sp}_* \G_{T_0 \cap \cdots \cap T_m}\O_{\P^{\rm ad}} \simeq {\rm sp}_* \mathcal C^\dag(T_0, \dots, T_m).$
    
    \begin{lemma}[{Berthelot, \cite[9.4.6]{intro}}] \label{rg_concentrate}
        In the situation of the previous lemma, suppose that locally there is a $p$-basis $t_i$ of $\P$ such that $T_i$ is cut out by $t_i$. Then
        $\H^{\dag i}_{T_0 \cap \cdots \cap T_m}(\O_{\P, \Q}) = 0$ if $i \neq m+1.$
    \end{lemma}
    
    \begin{lemma} \label{GY-resol}
        Let $E$ be an $\O_{\P^{\rm ad}}$-module. There is a canonical quasi-isomorphism of complexes
        $$\displaystyle [\G_{Z_{0 \cdots n}} E \rightarrow \cdots \rightarrow \bigoplus_{i_0 < i_1}\G_{Z_{i_0i_1}} E \rightarrow \bigoplus_{i_0}\G_{Z_{i_0}} E] \xrightarrow{\sim} \G_Y E,$$
        where on the left hand side, the right-most term is put at degree 0.
    \end{lemma}
    \begin{proof}
        Let $K_p$ be the complex
        $$\displaystyle \G_{Z_{0 \cdots p}}E \rightarrow \cdots \rightarrow \bigoplus_{0\leq i_0 < i_1 \leq p}\G_{Z_{i_0i_1}}E \rightarrow \bigoplus_{0 \leq i_0 \leq p}\G_{Z_{i_0}}E \rightarrow E.$$
        This lemma is equivalent to saying that $K_p$ is quasi-isomorphic to $j^\dag_{P - Z_0 \cup \cdots \cup Z_p} E.$ We proceed by induction on $p$. Observe that the mapping fiber of $\G_{Z_{p+1}} K_p \rightarrow K_p$ is isomorphic to $K_{p+1}[-1].$ This means that there is a distinguished triangle $\G_{Z_{p+1}} K_p \rightarrow K_p \rightarrow K_{p+1} \rightarrow$; by induction hypothesis, $K_{p+1}$ is quasi-isomorphic to $j^\dag_{P-Z_{p+1}} j^\dag_{P - Z_0 \cup \cdots \cup Z_p} E \simeq j^\dag_{P - Z_0 \cup \cdots \cup Z_{p+1}} E.$
    \end{proof}
    
    In the lemma above, each $\displaystyle \G_{Z_{i_0} \cap \cdots \cap Z_{i_p}}\O_{\P^{\rm ad}}$ of the complex on the left hand side has an ${\rm sp}_*$-acyclic resolution $\mathcal C^\dag_{i_0 \cdots i_p} := \mathcal C^\dag(T_{i_{p'} j})_{0 \leq p' \leq p, 1 \leq j \leq r}$ by Lemma \ref{sp_resol}. 
    We can formulate the double complex $\mathcal C^{\dag \cdot \cdot}$
    $$\displaystyle \mathcal C^\dag_{1 \cdots n} \rightarrow \cdots \rightarrow \bigoplus_{i_0 < i_1}\mathcal C^\dag_{i_0 i_1}\O_{\P^{\rm ad}} \rightarrow \bigoplus_{i_0}\mathcal C^\dag_{i_0}.$$
    Let $\mathcal C^{\dag \cdot}$ denote its total complex. $\mathcal C^{\dag \cdot}$ consists of ${\rm sp}_*$-acyclic terms and is isomorphic to $\G_Y\O_{\P^{\rm ad}}$ in the derived category; therefore ${\rm sp}_* \mathcal C^{\dag \cdot} \simeq \RG_Y\O_{\P, \Q}.$ We have a biregular spectral sequence
    $$\displaystyle E^{-p, q}_1 = \bigoplus_{i_0 < \cdots < i_p} \H^q( {\rm sp}_* \mathcal C^\dag_{i_0 \cdots i_p}) \Rightarrow \H^{\dag-p+q}_Y(\O_{\P, \Q}).$$
    Since $\mathcal C^\dag_{i_0 \cdots i_p}$ is an ${\rm sp}_*$-acyclic resolution of $\G_{Z_{i_0} \cap \cdots \cap Z_{i_p}}\O_{\P^{\rm ad}}$ by Lemma \ref{sp_resol}, we have $\displaystyle E^{-p, q}_1 \simeq \bigoplus_{i_0 < \cdots < i_p} \H^{\dag q}_{Z_{i_0\cdots i_p}}(\O_{\P, \Q}).$
    By our assumption, $Z_{i_0\cdots i_p}$ is locally cut out by some elements of finite $p$-basis of $\P$ and has codimension $r + p$. Therefore, Lemma \ref{rg_concentrate} implies that $E^{-p, q}_1 = 0$ if $-p + q \neq r.$ Thus the spectral sequence degenerates at $E_1$ and we have $\displaystyle E_1^{-p, q} = E_\infty^{-p, q} \simeq \bigoplus_{i_0 < \cdots < i_p} \H^{\dag q}_{Z_{i_0\cdots i_p}}(\O_{\P, \Q})$ and $\H^{\dag q}_{Z_{i_0\cdots i_p}}(\O_{\P, \Q}) = 0$ if $q \neq p + r.$
    Also, we have the descending filtration $F^p$ on $\H^{\dag r}_Y(\O_{\P, \Q})$ induced by the spectral sequence and we have by definition $$\displaystyle Gr_F^{-p} \H^{\dag r}_Y(\O_{\P, \Q}) \simeq E_\infty^{-p, r+p} \simeq \bigoplus_{i_0 < \cdots < i_p} \H^{\dag r+p}_{Z_{i_0\cdots i_p}}(\O_{\P, \Q}).$$
    Define $F_p := F^{-p}$; then $\displaystyle Gr^F_p \H^{\dag r}_Y(\O_{\P, \Q}) \simeq \bigoplus_{i_0 < \cdots < i_p} \H^{\dag r+p}_{Z_{i_0\cdots i_p}}(\O_{\P, \Q}).$
    
    Let us summarize the argument above as a theorem:
    \begin{theorem} \label{calc_rg}
        Let $\P$ be a projective, smooth formal scheme over $\V[[t]]$, $Z_0, \dots, Z_n$ closed subschemes of $P$ of codimension $r$ and let $Y = Z_0 \cup \cdots \cup Z_n$. For $0 \leq i_0 < \cdots < i_p \leq n$, we put $Z_{i_0 \cdots i_p} = Z_{i_0} \cap \cdots \cap Z_{i_p}$. We suppose that locally there exists a $p$-basis $t_i$ of $\P$ over $\S$ such that $Z_{i_0\cdots i_p}$ is defined by $t_{j_0 } = \cdots = t_{j_{p+r}} = 0$ for some $j_0 < \cdots < j_{p+r}$.
        
        Then there is an increasing filtration $F_p$ of $\H^{\dag r}_Y(\O_{\P, \Q})$ such that $F_{-1} = 0, F_n = \H^{\dag r}_Y(\O_{\P, \Q})$ and 
        
        $$\displaystyle Gr^F_p \H^{\dag r}_Y(\O_{\P, \Q}) \simeq \bigoplus_{i_0 < \cdots < i_p} \H^{\dag r+p}_{Z_{i_0\cdots i_p}}(\O_{\P, \Q}).$$
    \end{theorem}
    
    Suppose that $X$ is a strictly semistable scheme over $k[[t]]$ which is embedded as a closed subscheme of $P$ of codimension $r$. Let $D_1, \dots, D_n$ be the irreducible components of the closed fiber $X_s$. We can apply the theorem above to obtain a filtration $F_p'$ of $\H^{\dag r+1}_{X_s}(\O_{\P, \Q})$, where $r$ is the codimension of $X$ in $P$.
    
    Since $X_s = X \cap P_s$, there is a distinguished triangle $ \RG_{X_s} \O_{\P, \Q} \rightarrow \RG_{X} \O_{\P, \Q} \rightarrow \RG_{X} \O_{\P}(^\dag P_s)_{\Q} \rightarrow.$
    Taking the long exact sequence, we have an exact sequence
    $$0 \rightarrow \H^{\dag r}_X(\O_{\P, \Q}) \rightarrow \H^{\dag r}_X(\O_{\P}(^\dag P_s)_{\Q}) \rightarrow \H^{\dag r+1}_{X_s}(\O_{\P, \Q}) \rightarrow 0.$$
    We can thus define the (finite) filtration $F_p$ on $\H^{\dag r}_X(\O_{\P}(^\dag P_s)_{\Q})$ such that $F_{-1} = 0, F_0 = \H^{\dag r}_X(\O_{\P, \Q})$ and 
    $$\displaystyle Gr^F_p \H^{\dag r}_X(\O_\P(^\dag P_s)_\Q) \simeq \bigoplus_{i_1 < \cdots < i_p} \H^{\dag r+p}_{D_{i_1 \cdots i_p}}(\O_{\P, \Q}) \simeq \bigoplus_{i_1 < \cdots < i_p} \RG_{D_{i_1 \cdots i_p}}(\O_{\P, \Q})[r+p].$$
    Observe that $\RG_{D_{i_1\cdots i_p}}(\O_{\P, \Q})[r+p]$ is a $(\delta_X - p)$-shifted constructible sheaf on the frame $(D_{i_1\cdots i_p},D_{i_1\cdots i_p},\P)$. Indeed, if $u : \mathfrak D \hookrightarrow \P$ is a closed immersion lifting $D_{i_1\cdots i_p} \hookrightarrow P$, we have $\RG_{D_{i_1\cdots i_p}}(\O_{\P, \Q})[r+p] \simeq u_+u^! \O_{\P, \Q}[r+p] \simeq u_+\O_{\mathfrak D, \Q}$. Since $\O_{\mathfrak D, \Q}$ is a $\delta_D$-shifted constructible sheaf on $\mathfrak D$ and pushforward by closed immersion preserves constructible sheaves, we get the assertion.

    Now consider the distinguished trigangle $F_p/F_q \rightarrow F_{p+1}/F_q \rightarrow Gr_{p+1} \rightarrow$; if $F_p/F_q \in \C^{[-\delta_X + q+1, -\delta_X + p]}_{\rm cons}$, then since $Gr_{p+1} \in \C^{-\delta_X + p+1}_{\rm cons}$ and by the extension property Proposition \ref{ext-prop}, we inductively get $F_{p+1}/F_q \in \C^{[-\delta_X + q+1, -\delta_X + p+1]}_{\rm cons}$. 

    Therefore, $F_p \in \C^{\leq -\delta_X + p}_{\rm cons}$ and $F_\infty / F_p \in \C^{\geq -\delta_X + p+1}_{\rm cons}$, where $F_\infty = \H^{\dag r}_X(\O_{\P}(^\dag P_s)_{\Q})$. Considering the distinguished triangle $F_p \rightarrow F_\infty \rightarrow F_\infty/F_p \rightarrow$, we obtain $F_p \simeq \tau^{\leq -\delta_X + p}_{\rm cons} \H^{\dag r}_X(\O_{\P}(^\dag P_s)_{\Q})$. Since $\RG_{D_{i_1 \cdots i_p}}(\O_{\P, \Q})$ belongs to $D^b_{\rm hol}(X, X, \P)$, we deduce that $Gr^F_p \H^{\dag r}_X(\O_{\P}(^\dag P_s)_{\Q} \in {\rm Hol}(X, X, \P)$; by devissage, $F_p \in {\rm Hol}(X, X, \P)$.
    
    We summarize the above argument as the following theorem:
    
    \begin{theorem} \label{cons_trun}
        The filtration $F_p$ on $\H^{\dag r}_X(\O_{\P}(^\dag P_s)_{\Q})$ in ${\rm Hol}(X, X, \P)$ gives the constructible truncation $\tau^{\leq -\delta_X + p}_{\rm cons}$ of $\H^{\dag r}_X(\O_{\P}(^\dag P_s)_{\Q})$ in $D^b_{\rm hol}(X, X, \P).$
        
        $F_{-1} = 0, F_{\delta_X} = \H^{\dag r}_X(\O_{\P, \Q})$ and $\displaystyle Gr^F_p \simeq \bigoplus_{i_1 < \cdots < i_p} \RG_{D_{i_1 \cdots i_p}}(\O_{\P, \Q})[r+p].$
    \end{theorem}
    
    \begin{remark} \label{cons_trun_form}
        Before ending this subsection, we remark that when there is a lift $\X$ of $X$ to a formal scheme over $\V[[t]]$ locally having finite $p$-basis over $\V$, the filtration $F_p$ on $\O_{\X}(^\dag X_s)_{\Q}$ defined above is given by $F_p = \sum_{i_1 < \cdots < i_p} \O_{\X}(^\dag D_{i_1} \cup \cdots \cup D_{i_p})_{\Q}$. To see this, recall that the spectral sequence is obtained from the double complex
        $$ {\rm sp}_* \mathcal{C}^{\dag \cdot}(D_1, \dots, D_n) \rightarrow \cdots \rightarrow \bigoplus_{i_0} \mathcal {\rm sp}_* \mathcal{C}^{\dag \cdot}(D_{i_0}).$$
        Since ${\rm sp}_* \mathcal{C}^{\dag \cdot}(D_{i_0}, \dots, D_{i_m}) = [\O_{\X, \Q} \rightarrow \bigoplus_{j_0} \O_{\X}(^\dag D_{i_{j_0}})_\Q \rightarrow \cdots \rightarrow \O_{\X}(^\dag D_{i_0} \cup \cdots \cup D_{i_m})_\Q]$, we get $F_p = \sum_{i_1 < \cdots < i_p} \O_{\X}(^\dag D_{i_1} \cup \cdots \cup D_{i_p})_{\Q}.$
        \\
        Fix a data of lifts $(\X_\alpha, \P_\alpha, u)$ of the frame $(X, X, \P)$. By definition, we have $\varepsilon : u_0^* \RG_{X_\eta}(\O_{\P, \Q})[r] \xrightarrow{\sim} (\O_{\X_\alpha}(^\dag X_{\alpha, s})_\Q, \star)$ where $\star$ is the canonical isomorphism induced by $\tau$. Since $u_0^*$ preserves constructible t-structure in the sense of Proposition \ref{u_0^*-cons-t}, by Theorem \ref{cons_trun}, the $\varepsilon$ induces an isomorphism $u_0^* F_p \RG_{X_\eta}(\O_{\P, \Q})[r] \simeq (F_p \O_{\X_\alpha}(^\dag X_{\alpha, s})_\Q, \star).$
    \end{remark}
    
    \subsection{Purity theorem}
    Let $\pi : (X, Y, \P) \rightarrow (S, S, \S)$ be a morphism of frames. We define the unit of tensor products to be $\one_{(X, Y, \P)} := \pi^+ K \in D^b_{\rm hol}(X, Y, \P)$, where $K$ is regarded as an object of $D^b_{\rm hol}(S, S, \S).$ Note that since $K$ is a constructible sheaf, so is $\one_{(X, Y, \P)}.$
    By definition, we have $\one_{(X, Y, \P)} = \DD_{(X, Y, \P)}\pi^!\DD_{(S, S, \S)}K = \DD_{(X, Y, \P)}(\RG_{X}\O_{\P, \Q}[\delta_X]).$
    If $Y$ locally has finite $p$-basis over $k$, is of codimension $r$ in $P$ and there is a divisor $T$ of $P$ such that $Y - X = Y \cap T$, then since $\DD_\P(\RG_{Y}(\O_{\P, \Q})[r]) \simeq \RG_{Y}(\O_{\P, \Q})[r](-\delta_Y)$ by [Car21, 9.4.11] and [Abe14, 3.14] for Frobenius, we have
    
        \begin{align} \label{one_calc}
            \one_{(X, Y, \P)} &= (^\dag T)\DD_\P(\RG_{Y}(\O_{\P, \Q})[r][\delta_P - r]) \\
            &\simeq (^\dag T)\RG_{Y}(\O_{\P, \Q})(-\delta_Y)[r][r - \delta_P] \notag \\
            &\simeq \RG_{X}\O_{\P, \Q}(-\delta_Y)[2r-\delta_P]. \notag
        \end{align}
    
    \begin{proposition} \label{purityA}
        Let $X$ be a scheme locally having finite $p$-basis over $k$ which embeds as a closed subscheme into a projective formal scheme $\P$ over $\V[[t]]$ with locally finite $p$-basis over $\V$ and let $\pi : (X, X, \P) \rightarrow (S, S, \S)$ be the morphism of frames. Then $\one_{(X,X,\P)} \simeq \pi^! \one_S(-\delta_X)[-2\delta_X].$
    \end{proposition}
    \begin{proof}
        Let $f : \P \rightarrow \S$ be the structure morphism. Let $r = \delta_P - \delta_X$ be the codimension of $X$ in $P$. Then
        $$\one_{(X, X, \P)} \simeq \RG_{X}\O_{\P, \Q}(-\delta_X)[2r - \delta_P],$$
        $$\pi^!\one_S(-\delta_X)[-2\delta_X] = \RG_X f^! K (-\delta_X)[-2\delta_X]= \RG_X \O_{\P, \Q}(-\delta_X)[-2\delta_X + \delta_P].$$
        Noting that $2r-\delta_P = -2\delta_X + \delta_P$, we get the proposition.
    \end{proof}
    
    In the rest of this subsection, let $X$ be a strictly semistable scheme over $k[[t]]$ which embeds as a closed subscheme into a projective formal scheme $\P$ over $\V[[t]]$ with locally finite $p$-basis over $\V$.
    Let $D_0, \dots, D_r$ be the irreducible components of $X_s$. For $I \subseteq \{0, \dots, r\},$ we put $D_I := \cap_{i \in I} D_i.$
    \\
    Let $\pi : (X, X, \P) \rightarrow (S, S, \S)$ be the structure morphism, $j : (X_\eta, X, \P) \hookrightarrow (X, X, \P)$ the open immersion, and let $i: (X_s, X_s, \P) \hookrightarrow (X, X, \P)$ and $i_I : (D_I, D_I, \P) \hookrightarrow (X, X, \P)$ be the closed immersions.
    
    \begin{proposition} \label{purityB}
        Let $I \subseteq \{0, \dots, r\}$ with $\card I = p+1.$
        Then there is an isomorphism $$\one_{D_I} \simeq i_I^!\one_X(p+1)[2(p+1)].$$
    \end{proposition}
    \begin{proof}
        Using Proposition \ref{purityA}, we have
        $$\one_{D_I} \simeq (fi_I)^! \one_S(-\delta_X + p+1)[2(-\delta_X+p+1)],$$
        $$i_I^!\one_X(p+1)[2(p+1)] \simeq i_I^!f^!\one_S(-\delta_X+p+1)[-2\delta_X + 2(p+1)].$$
    \end{proof}

    \begin{proposition} \label{purityC}
        We have an isomorphism $\displaystyle \H^{p+2}_{\rm cons}(i_+i^!\one_X) \simeq \bigoplus_{\card I = p+1} i_{I+}i_I^!\one_X[2p+2]$ for $p \geq 0$.
    \end{proposition}
    \begin{proof}
        We have $i_+i^!\one_X = \RG_{X_s}\O_{\P, \Q}(-\delta_X)[r-\delta_X], \ i_{I+}i_I^!\one_X = \RG_{D_I}\O_{\P, \Q}(-\delta_X)[r-\delta_X].$
        By Theorem \ref{cons_trun}, we get $$\displaystyle \H^{p+2}_{\rm cons}(i_+i^!\one_X) = \H^{p+2}_{\rm cons}(\RG_{X_s}\O_{\P, \Q}[r-\delta_X]) \simeq \bigoplus_{\card I = p+1} \RG_{D_I}\O_{\P, \Q}[r-\delta_X+2p+2].$$
    \end{proof}
    
    \begin{proposition} \label{purityD}
        We have an isomorphism $\H^{p+1}_{\rm cons}(j_+\one_{X_\eta}) \simeq \H^{p+2}_{\rm cons}(i_+i^!\one_X)$ for $p \geq 0$.
    \end{proposition}
    \begin{proof}
            Consider the long exact sequence of constructible sheaves associated to the basic distinguished triangle
            $$i_+i^!\one_X \rightarrow \one_X \rightarrow j_+j^+\one_X \rightarrow.$$
            Note that $\one_X$ is a constructible sheaf and $j^+\one_X = \one_{X_\eta}$.
    \end{proof}
    
    \begin{theorem} {\rm (purity)}\label{purity}
        $\displaystyle\bigoplus_{\card I = p+1} i_{I+} \one_{D_I} \simeq \H^{p+1}_{\rm cons}(j_+\one_{X_\eta})(p+1)$ for $p \geq 0$.
    \end{theorem}
    \begin{proof}
        Using Propositions \ref{purityB}, \ref{purityC} and \ref{purityD}, we have
        \begin{eqnarray*}
            \bigoplus_{\card I = p+1} i_{I+} \one_{D_I} & \simeq & \displaystyle\bigoplus_{\card I = p+1} i_{I+} i_I^!\one_X(p+1)[2p+2]\\
            & \simeq & \H^{p+2}_{\rm cons}(i_+i^!\one_X)(p+1)\\
            & \simeq & \H^{p+1}_{\rm cons}(j_+\one_{X_\eta})(p+1).
        \end{eqnarray*}
    \end{proof}
    
    \section{Nearby cycle} \label{sec-nc}
    In this section, we let $X$ be a strictly semi-stable scheme over $k[[t]]$ which embeds into a smooth projective formal scheme over $\P$ as a closed subscheme. Let $\pi_\eta : (X_\eta, X, \P) \rightarrow (\eta, \DD_S, \DD_\S)$ be the structure morphism. Let $\O_{X_\eta} := \RG_{X_\eta} \O_{\P, \Q}[\delta_P - \delta_X] \in {\rm Hol}(X_\eta, X, \P).$
    \\
    Let $D_0, \dots, D_r$ denote the irreducible components of the closed fiber $X_s$. For $I \subseteq \{ 0, \dots, r\},$ we define $D_I := \bigcap_{i \in I} D_i.$ 
    Let $j : X_\eta \rightarrow X$ be the open immersion. Note that $j_!$ and $j_+$ are exact functors from ${\rm Hol}(X_\eta, X, \P)$ to ${\rm Hol}(X, X, \P)$ since $j_+$ is the inclusion functor and $j_! = \DD_\P (^\dag P_s) \DD_\P$ by definition.
    
    For the construction of the (unipotent) nearby cyle $\Psi$, we follow the paper [AC14]. In the case over $k[[t]]$, the construction $\Psi \E$ for general $\E$ does not work for the moment; the proof of the key lemma $\displaystyle \lim_{\longleftrightarrow} j_! \E^{\cdot, \cdot} \xrightarrow{\sim} \lim_{\longleftrightarrow} j_+ \E^{\cdot, \cdot}$ of [AC14, 2.4] uses Kedlaya's semistable reduction theorem, which is not known in our context. However, we can show that the nearby cycle of $\O_{X_\eta}$ can be defined by locally investigating the intermediate extension $j_{!+}I^{0, N}_X$.
    
    \subsection{Construction of $\Psi$} \label{nc-construct}
    We recall the $\D$-module $I = ``t^x"$ from \cite{bei}.
    
    Let $\O_\eta := \O_{\DD_{\S}}(^\dag s)_\Q, \ \D_\eta := \O_\eta \otimes_{\O_{\DD_\S,\Q}} \D_{\DD_\S}.$ We introduce the $\D_\eta$-module $I_{\DD_\S} := \O_\eta[x, x^{-1}]\cdot t^x$ by
    $$\partial_t x^k t^x = t^{-1} x^{k+1} t^x.$$
    For $a \in \mathbb{Z}$, put $I^a_{\DD_\S} := \O_\eta[x]x^a \cdot t^x$ and for $a \leq b \in \mathbb{Z}$, put $I^{a,b}_{\DD_\S} := I^a_{\DD_\S} / I^b_{\DD_\S}.$ The structure of $\D_\eta$-module on $I^{a,b}_{\DD_\S}$ extends uniquely by continuity to the structure of $\D^\dag_{\DD_\S}(^\dag s)_\Q$-module. 

    The map $I^{a,b}_{\DD_S} \rightarrow F^*I^{a,b}_{\DD_S}, \ gx^kt^x \mapsto q^kg \otimes x^k t^x$ is a $\D^\dag_{\DD_S}(^\dag s)_\Q$-linear isomorphism and therefore defines a Frobenius structure on $I^{a,b}_{\DD_S}.$ $I^{a,b}_{\DD_S}$ is an overconvergent $F$-isocrystal on $(\eta, \DD_S, \DD_\S)$ in the sense that it is coherent over $\O_\eta$ and admits a Frobenius structure. Noting that $I^{a,b}$ is an iterated extension by $\O_\eta \in {\rm Hol}(\eta),$ we have $I_{\DD_S}^{a,b} \in {\rm Hol}(\eta)$.

    Multiplying by $x^n$ induces morphisms of $F$-$\D^\dag_{\DD_S}(^\dag s)_\Q$-modules $I^{a,b}_{\DD_S} \rightarrow I^{a,b}_{\DD_S}(-n)$ for $n \leq 0$ and $I^{a,b}_{\DD_S} \xrightarrow{\sim} I^{a+n,b+n}_{\DD_S}(-n)$ for any $n \in \mathbb Z$.

    There is a perfect pairing $I^{a,b}_{\DD_S} \tilde \otimes_\eta I^{-b,-a}_{\DD_S} \rightarrow \O_\eta(-1), \ (f(x), g(x)) \mapsto {\rm Res}_{x=0} f(x)g(-x).$ This induces an isomorphism $\DD_\eta(I^{a,b}_{\DD_S}) \simeq I^{-b,-a}_{\DD_S}$. Moreover, this pairing induces a commutative diagram
    \begin{equation} \label{dual-x}
    \xymatrixcolsep{20mm}
    \xymatrix{
        \DD I_{\DD_S}^{a, b} \ar[r]^{\DD x} \ar[d]_{\wr} & \DD I_{\DD_S}^{a, b}(-1) \ar[d]_{\wr}\\
        I_{\DD_S}^{-b, -a} \ar[r]^{x} & I_{\DD_S}^{-b, -a}(-1). \\
    }
    \end{equation}
    \noindent
    We put $I^{a,b}_X := \pi_\eta^+(I^{a,b}_{\DD_S})[\delta_X - 1] \in {\rm Hol}(X_\eta).$ $I^{a,b}_X$ is a holonomic module on $X_\eta$ and $\pi_\eta^! I^{a,b}_{\DD_S} \simeq \pi_\eta^+I^{a,b}_{\DD_S}(\delta_X-1)[2\delta_X-2]$ (see \cite[5.5]{frob}, which holds in our setting). By \cite[3.10]{frob} (which also holds in our setting), we also have
    \begin{align}
        \DD_{X_\eta}(I_X^{a,b}) & = \DD_{X_\eta}\pi_\eta^+ (I_{\DD_S}^{a,b})[1-\delta_X] \notag\\
        & \simeq \pi_\eta^+ \DD_{\eta}(I_{\DD_S}^{a,b})(\delta_X-1)[\delta_X-1] \notag\\
        & \simeq \pi_\eta^+ I_{\DD_S}^{-b, -a}(\delta_X-1) = I_X^{-b, -a}(\delta_X-1).\notag
    \end{align}
    
    \begin{remark} \label{ix-loc}
        Suppose that there exist an open subset $\U \subseteq \P$, a lift $\X \rightarrow \DD_\S$ of $X \cap U \rightarrow \DD_\S$ such that $\X$ locally has finite $p$-basis over $\S$ and a closed immersion $u : \X \hookrightarrow \U$ which lifts $X \cap U \hookrightarrow U$. Then $u^!(I^{a,b}_X|_\U)$ is represented by the $\D^\dag_\X(^\dag X_s \cap U)_\Q$-module $\displaystyle \bigoplus_{a \leq k < b} \O_\X(^\dag X_s)_\Q \, x^k t^x$; for a vector field $P$ of $\D^\dag_\X(^\dag X_s \cap U)_\Q$, we have  $P(gx^k t^x) = P(g)x^kt^x + gP(t)t^{-1}x^{k+1}t^x.$ The Frobenius structure $I_X^{a, b} \xrightarrow{\sim} F^*I_X^{a, b}$ is given by $gx^kt^x \mapsto q^{k + 1-\delta_X} g \otimes x^k t^x.$
    \end{remark}
    
    \begin{definition}[{\cite[2.5]{bei}}] \label{un-def}
        For $\E \in {\rm Hol}(X_\eta, X, \P),$ we define $\E^{a, b} := \E \otimes I_X^{a,b}$ and for any $k \in \mathbb{Z}$, $\E_k^{a, b} := \E^{\max(a, k), \max(b, k)}.$
        \\
        The (unipotent) nearby cycle is defined to be $\displaystyle \Psi \E := \coker(\lim_{\longleftrightarrow} j_!\E_0^{\cdot, \cdot} \rightarrow \lim_{\longleftrightarrow} j_+\E_0^{\cdot, \cdot}) \in \lim_{\longleftrightarrow} {\rm Hol}(X, X, \P).$
    \end{definition}
    
    \begin{proposition} \label{ie_char}
        Let $j : (X_\eta, X, \P) \hookrightarrow (X, X, \P)$ be the open immersion of frames. As explained in the beginning of this section, $j_!$ and $j_+$ define exact functors $F{\rm \mathchar`-Ovhol}(X_\eta, X, \P) \rightarrow F{\rm \mathchar`-Ovhol}(X, X, \P)$.
        Intermediate extension $j_{!+} \E := \im(j_! \E \rightarrow j_+ \E)$ has the following property: it is the smallest subobject $\E'$ of $j_+ \E$ in $F{\rm \mathchar`-Ovhol}(X, X, \P)$ such that $j^! \E' = \E.$
    \end{proposition}
    \begin{proof}
        Suppose $j^! \E' = \E$ and put $\E'' = j_{!+} \E \cap \E'$. By the left t-exactness of $j^!$, we have $j^!(\E'') = j^!(j_{!+}\E) \cap j^!(\E') = \E.$ Put $\mathcal F = j_{!+} \E / \E''$. We have the inclusion
        $$\DD_X (\F) \subseteq \DD_X(j_{!+} \E) \simeq j_{!+} \DD_{X_\eta} (\E) \subseteq j_+ \DD_{X_\eta}(\E).$$
        Since $\F$ is supported on $X_s$, so is $\DD_X(\F)$. Hence $\DD_X(\F) \simeq \H^0(\RG_{X_s}\DD_X(\F)) \subseteq \H^0(\RG_{X_s}j_+\DD_{X_\eta}(\E)) = 0.$ Thus $\F = 0$.
    \end{proof}
    \begin{proposition} \label{ie_form}
        In the situation of Remark \ref{ix-loc}, we have
        $$\displaystyle u^!(j_{!+} I^{0, N}_X|_\U) = \D^\dag_\X(^\dag X_s \cap U)_\Q \, t^x = \bigoplus_{k = 0}^{N-1} F_k\O_{\X}(^\dag X_s)_{\Q} \, x^k t^x.$$
    \end{proposition}
    \begin{proof}
        The question is local and we assume that $\U = \P$ and $\X$ are affine. In this situation, $I_X^{a,b}$ is represented by the coherent $\D^\dag_\X(^\dag X_s)_\Q$-module $\displaystyle \bigoplus_{a \leq k < b} \O_\X(^\dag X_s)_\Q \, x^k t^x.$
        
        (i) Let us show the equality on the right hand side. Recall that $F_p \O_\X(^\dag X_s)_\Q = \sum_{j_1 < \dots < j_p} \O_\X(^\dag D_{j_1} \cup \dots \cup D_{j_p})_\Q$ by Remark \ref{cons_trun_form}. Note that $\bigoplus_k F_k\O_{\X}(^\dag X_s)_{\Q} \, x^k t^x$ is a $\D^\dag_{\X, \Q}$-module. Indeed, if we take a $p$-basis $t_0, \dots, t_n$ with $t = t_0 \cdots t_r$, we have $\partial_{t_i} (gx^k t^x) = \partial_{t_i}(g) x^k t^x + gt_i^{-1} x^{k+1} t^x$ for $i \leq r$ and $\partial_{t_i} (gx^k t^x) = \partial_{t_i}(g) x^k t^x$ for $i > r.$ If $g \in F_k,$ we have $\partial_{t_i}(g) t_i^{-1} \in F_{k+1}$ for $i \leq r$ (see Remark \ref{cons_trun_form}). It follows that $\D^\dag_{\X, \Q} \, t^x \subseteq \bigoplus_k F_k\O_{\X}(^\dag X_s)_{\Q} \, x^k t^x.$ 
        
        Let us show the converse inclusion. Remark that $\O^\dag_\X(^\dag D_{i_1} \cup \cdots \cup D_{i_p})_\Q$ is generated by $t_{i_1}^{-1} \cdots t_{i_p}^{-1}$ as a $D^\dag_{\X, \Q}$-module (\cite[9.4.6]{intro}). Let us show $F_{N-1} \O_{\X}(^\dag X_s)_\Q \, x^{N-1} t^x \subseteq \D^\dag_{\X, \Q} \, t^x.$
        We suppose $r+1 \leq N-1$ (the case of $N-1 < r+1$ is quite similar).
        Put $\zeta := \partial_{t_0}\cdots \partial_{t_r}(t^x) = t_0^{-1}\cdots t_r^{-1} x^{r+1} t^x = t^{-1}x^{r+1}t^x$; we have $(-t_0\partial_{t_0})^{k}(\zeta) = t^{-1} x^{r+1+k} t^x.$ For any $P \in \D^\dag_{\X, \Q},$ we have $P(t^{-1} x^{N-1}) = P(t^{-1}) x^{N-1}$; hence $\O_\X(^\dag X_s)_\Q x^{N-1} t^x \subseteq \D^\dag_{\X, \Q} \, t^x.$ Suppose by induction $\O_\X(^\dag X_s)_\Q x^{r+1+l} t^x \subseteq \D^\dag_{\X, \Q} \, t^x$ for all $l > k.$ Then for any $P \in \D^\dag_{\X, \Q},$ we have $P(t^{-1} x^{r+1+k} t^x) = P(t^{-1}) x^{r+1+k} t^x + ({\rm higher \  terms}).$ As the left hand side and the higher terms belong to $\D^\dag_{\X, \Q} \, t^x$, so is $P(t^{-1}) x^{r+1+k}.$
        This proves that for $k \geq r+1$, $\O_{\X}(^\dag X_s)_\Q \,t^k x^k \subseteq \D^\dag_{\X, \Q} \, t^x.$ Let $k \leq r$. Suppose by induction for $l > k$, $F_l \O_{\X}(^\dag X_s)_\Q \, t^l t^x \subseteq \D^\dag_{\X, \Q} \, t^x.$ For $0 \leq i_1 < \cdots < i_k \leq r, \ f \in \O_\X(^\dag D_{i_1} \cup \cdots \cup D_{i_p} )_\Q$, let $f = P(t_{i_1}^{-1} \cdots t_{i_p}^{-1})$ for some $P \in \D^\dag_{\X, \Q}.$ Then $P\partial_{t_{i_1}}\cdots\partial_{t_{i_p}}(t^x) = P(t_{i_1}^{-1} \cdots t_{i_p}^{-1} t^k t^x) = P(t_{i_1}^{-1} \cdots t_{i_p}^{-1}) x^kt^k + ({\rm higher \ terms})$. Similarly we get $fx^kt^x \in \D^\dag_{\X, \Q} \, t^x.$
        
        (ii) By the similar argument, using the fact that $\O_\X(^\dag X_s)_\Q$ is generated by $t^{-1}$ as a $\D^\dag_{\X, \Q}$-module, we can show that $I_X^{0, N}$ is generated by $t^{-1}t^x$ as a $\D^\dag_{\X, \Q}$-module. Thus $j^! (\D^\dag_{\X, \Q} \, t^x) = I_X^{0, N}$ and therefore $j_{!+} I_X^{0, N} \subseteq \D^\dag_{\X, \Q} \, t^x$. Let us prove $t^x \in j_{!+} I_X^{0, N}$. As explained in subsection \ref{ssss}, we may assume that $X$ is smooth over $\Spec k[[t]][t_0, \dots, t_r]/(t - t_0 \cdots t_r)$ and that $\X$ is smooth over $\Spf \V[[t]]\langle t_0, \dots t_r \rangle/(t - t_0 \cdots t_r)$. Since the pullback functors of smooth morphisms are exact after shifting by relative dimensions, we may assume that $\X = \Spf \V[[t]]\langle t_0, \dots t_r \rangle/(t - t_0 \cdots t_r).$ More precisely, if $g : \X \rightarrow \Spf \V[[t]]\langle t_0, \dots, t_r \rangle/(t - t_0 \cdots t_r) := \X'$ denotes the smooth morphism of relative dimension $d$, then $g^+[d] \simeq g^![-d] : {\rm Ovhol}(\X') \rightarrow {\rm Ovhol}(\X)$ and $g_\eta^+[d] \simeq g_\eta^![-d] : {\rm Ovhol}(X'_\eta, X', \X') \rightarrow {\rm Ovhol}(X_\eta, X, \X)$ are exact functors and $g^+ j_! [d] \simeq j_! g_\eta^+ [d] $ and $g^+ j_+ [d] \simeq j_+ g_\eta^+ [d]$. Therefore we have $g^+j'_{!+}[d] = j_{!+}g_\eta^+[d]$ and in order to show $t^x \in j_{!+} I_X^{0, N}$, it suffices to show $t^x \in j'_{!+} I_{X'}^{0, N}$.
         
        (ii-a) Let us show the following statement: any element $f$ of $\Gamma(\X, \O_\X(^\dag X_s))$ is uniquely of the form $f_+ + \sum_{\underline\nu \in \mathbb{N}^{r+1}} a_{\underline\nu} t_0^{-\nu_0} \cdots t_r^{-\nu_r}$, where $f_+ \in \Gamma(\X, \O_\X)$ and $a_{\underline\nu} \in \V$ such that $|a_{\underline\nu}| < c\eta^{|\underline\nu|}$ for some constants $c > 0$ and $0 < \eta < 1$. Moreover in this situation, $f \in \Gamma(\X, \O_\X(^\dag D_{i_0} \cup \dots \cup D_{i_k}))$ if and only if $a_{\underline\nu} = 0$ if $\nu_j > 0$ for some $j \not\in \{i_0, \dots, i_k\}$.
        
        Recall that $f \in \Gamma(\X, \O_\X(^\dag X_s))$ is of the form $\sum_{\underline\nu} f_{\underline\nu} t_{0}^{-\nu_0} \cdots t_{p}^{-\nu_p}$ for some $f_{\underline\nu} \in \Gamma(\X, \O_\X)$ such that $|f_{\underline\nu}| < c \eta^{|\underline\nu|}$ for some constants $c > 0$ and $0 < \eta < 1$. Let $f_{\underline\nu} = \sum_{\underline\mu \in \mathbb{N}^{r+1}} f_{\underline\nu, \underline{\mu}} \, t_0^{\mu_0} \cdots t_r^{\mu_r}$ with $f_{\underline\nu, \underline{\mu}} \in \V[[t]]$ and $|f_{\underline\nu, \underline{\mu}}| < c \eta^{|\underline\nu|}$, $|f_{\underline\nu, \underline{\mu}}| \rightarrow 0$ as $|\underline{\mu}| \rightarrow \infty$. Let $f_{\underline\nu, \underline{\mu}} = \sum_n f_{\underline\nu, \underline{\mu}, n} t^n$ with $f_{\underline\nu, \underline{\mu}, n} \in \V$. We have $|f_{\underline\nu, \underline{\mu}, n}| < c \eta^{|\underline\nu|}$. Using $f_{\underline\nu, \underline{\mu}, n}$, we get $f = \sum_{\underline\nu, \underline{\mu}, n} f_{\underline\nu, \underline{\mu}, n} t^n t_0^{\mu_0 - \nu_0} \cdots t_r^{\mu_r - \nu_r}.$
        Noting that $t = t_0 \cdots t_r$ in $\Gamma(\X, \O_\X)$, for $\underline\nu \in \mathbb{N}^{p+1}$, we put $a_{\underline\nu} := \sum_{n, \underline{\mu}} f_{\underline n + \underline\mu + \underline\nu, \underline{\mu}, n}$, where $\underline n = (n, \dots, n)$. This sum converges in $\V$ since $|f_{\underline n + \underline\mu + \underline\nu, \mu, n}| < c \eta^{n(r+1) + |\underline\mu| + |\underline\nu|}$. Obviously $|a_{\underline\nu}| < c\eta^{|\underline\nu|}$. To see that $f_+ := f - \sum_{\underline\nu} a_{\underline\nu} \underline{t}^{-\underline\nu}$ belongs to $\Gamma(\X, \O_{\X}),$ it suffices to show it modulo $p^{m}$ for all $m \geq 1$; in this case all sums are finite and using $t = t_0 \cdots t_r$, it is easy to check. The uniqueness is also checked modulo $p^m$. If $f \in \Gamma(\X, \O_\X(^\dag D_{i_0} \cup \dots \cup D_{i_k}))$, we may assume $f_{\underline\nu} = 0$ if $\nu_j > 0$ for some $j \not\in \{i_0, \dots, i_k\}$. Therefore $a_{\underline\nu} = 0$ if $\nu_j > 0$ for some $j \not\in \{i_0, \dots, i_k\}$.
        
        (ii-b) Define $\mathcal{A}_k := \{ f \in \O_{\X}(^\dag X_s)_\Q : f x^k t^x + x^{k+1} \xi \in j_{!+} I_X^{0, N} \ {\rm for \ some} \ \xi \in I_X^{0, N}\}.$ This is a $F$-$\D^\dag_{\X, \Q}$-submodule of $\O_{\X}(^\dag X_s)_\Q$ (with some suitable Tate twist). Indeed, it is the image of the homomorphism $I_X^{k, N} \cap j_{!+} I_X^{0, N} \hookrightarrow I_X^{k, N} \rightarrow I_X^{k, k+1} \simeq \O_\X(^\dag X_s)_\Q$. Let us show $\mathcal{A}_k = F_k \O_\X(^\dag X_s)_\Q$. Since $j_{!+}I_X^{0,N} \subseteq \D^\dag_\X(^\dag X_s)_\Q \, t^x = \bigoplus_k F_k \O_{\X}(^\dag X_s)_\Q \, x^kt^x$, we get $\mathcal{A}_k \subseteq F_k \O_\X(^\dag X_s)_\Q$. We will show the converse inclusion by induction on $k$.
        Let $p : I_X^{0, N} \rightarrow I_X^{0, 1} \simeq \O_{\X}(^\dag X_s)_\Q$ be the canonical homomorphism. By abuse of notation we also let $p$ denote the homomorphism $j_+(p) : j_+I_X^{0, N} \rightarrow j_+I_X^{0, 1}$. Since the characteristic cycle ${\rm ZCar}(\O_{\X, \Q})$ of $\O_{\X, \Q}$ is $[X]$, where $[X]$ is viewed as an algebraic cycle of the tangent space $TX$, $\O_{\X, \Q}$ is an irreducible object of $F\mathchar`-{\rm Ovhol}(\X).$ Therefore $j_{!+} \O_{\X}(^\dag X_s)_\Q = \O_{\X, \Q}.$ Since $p$ maps $j_{!+} I_X^{0, N}$ into $j_{!+} \O_{\X}(^\dag X_s)_\Q = \O_{\X, \Q}$, $p(j_{!+} I_X^{0, N})$ is $0$ or $\O_{\X, \Q}.$ As $j^! = (^\dag X_s)$ is an exact functor, we have $(^\dag X_s)p(j_{!+} I_X^{0, N}) = p((^\dag X_s) j_{!+}I_X^{0, N}) = p(I_X^{0, N}) \neq 0$ and hence $\mathcal{A}_0 = p(j_{!+} I_X^{0, N}) = \O_{\X, \Q} = F_0 \O_\X(^\dag X_s)_\Q.$ Suppose that the assertion holds for $k$. Since multiplication by $x$ preserves $j_{!+} I_X^{0, N}$, $F_k \O_\X(^\dag X_s)_\Q = \mathcal {A}_k \subseteq \mathcal {A}_{k+1}$. The characteristic cycle of $\H^{\dag k+1}_{D_{i_0} \cap \dots \cap D_{i_k}}(\O_{\X, \Q})$ is $[D_{i_0} \cap \dots \cap D_{i_k}]$, so it is an irreducible object of $F\mathchar`-{\rm Ovhol}(\X)$. By the isomorphism of Theorem \ref{calc_rg}, $F_{k+1} \O_\X(^\dag X_s)_\Q / F_k \O_\X(^\dag X_s)_\Q$ is a semi-simple object of $F\mathchar`-{\rm Ovhol}(\X)$. Since the characteristic cycles of each irreducible constituents are all distinct, each of which is not isomorphic to each other (as $F\mathchar`-\D^\dag$-modules). Hence in order to prove that the subobject $\mathcal{A}_{k+1} / F_{k} \O_\X(^\dag X_s)_\Q$ of $F_{k+1} \O_\X(^\dag X_s)_\Q / F_k \O_\X(^\dag X_s)_\Q$ is equal to $F_{k+1} \O_\X(^\dag X_s)_\Q / F_k \O_\X(^\dag X_s)_\Q$, it suffice to show by considering the irreducible constituents that the projection $F_{k+1} \O_\X(^\dag X_s)_\Q / F_k \O_\X(^\dag X_s)_\Q \rightarrow \H^{\dag k+1}_{D_{i_0} \cap \dots \cap D_{i_k}}(\O_{\X, \Q})$ restricted to  $\mathcal{A}_{k+1} / F_{k} \O_\X(^\dag X_s)_\Q$ is a nonzero homomorphism. Let us find some element of $\mathcal{A}_{k+1}$ which maps to a nonzero element. There exists an $f \in \Gamma(X, F_{k+1}\O_{\X}(^\dag X_s)_\Q)$ such that $t_{i_1}^{-1} \cdots t_{i_k}^{-1} x^k t^x + f x^{k+1} t^x + ({\rm higher \ terms}) \in j_{!+}I_X^{0, N}.$ We have
        $$\partial_{i_0}(t_{i_1}^{-1} \cdots t_{i_k}^{-1} x^k t^x + fx^{k+1}t^x + ({\rm higher})) = (t_{i_0}^{-1} \cdots t_{i_k}^{-1} + \partial_{i_0}(f))x^{k+1}t^x + ({\rm higher}).$$
        By (ii-a), we have $f = p^{-n}(f_+ + \sum_{\underline\nu} a_{\underline\nu} t_0^{-\nu_0} \cdots t_r^{-\nu_r})$ for some $n \geq 0, \ f_+ \in \Gamma(\X, \O_\X), \ a_{\underline\nu} \in \V$ such that $|a_{\underline\nu}| < c\eta^{|\underline\nu|}$ for some constants $c > 0$ and $0 < \eta < 1$. We have
        $$t_{i_0}^{-1} \cdots t_{i_k}^{-1} + \partial_{i_0}(f) = p^{-n}(\partial_{i_0}(f_+) + p^nt_{i_0}^{-1} \cdots t_{i_k}^{-1} - \sum_{\nu_{i_0} \neq 0} \nu_{i_0}a_{\underline\nu} t_0^{-\nu_0} \cdots t_{i_0}^{-\nu_{i_0}-1} \cdots t_r^{-\nu_r}).$$
        Again by (ii-a), we conclude that $t_{i_0}^{-1} \cdots t_{i_k}^{-1} + \partial_{i_0}(f) \in \mathcal{A}_{k+1}$ maps to a nonzero element of $\H^{\dag k+1}_{D_{i_0} \cap \dots \cap D_{i_k}}(\O_{\X, \Q})$. To see this, recall that the map $\H^{\dag k+1}_{D_{j_0} \cap \dots \cap D_{j_k}}(\O_{\X, \Q}) \hookrightarrow F_{k+1} \O_\X(^\dag X_s)_\Q / F_{k} \O_\X(^\dag X_s)_\Q$ is induced by the canonical map $ \O_\X(^\dag D_{j_0} \cup \dots \cup D_{j_k})_\Q / \sum_l \O_\X(^\dag D_{j_0} \cup \dots \cup \hat{D}_{j_l} \cup \dots \cup D_{j_k})_\Q \hookrightarrow F_{k+1} \O_\X(^\dag X_s) / F_{k} \O_\X(^\dag X_s)$. Thus by induction we get $\mathcal{A}_k = F_k \O_\X(^\dag X_s)_\Q$. We can now prove by descending induction on $k$ that $\mathcal{A}_k x^k t^x = F_k \O_\X(^\dag X_s)_\Q x^k t^x \subseteq j_{!+} I_X^{0, N}$ as follows: for any $f_k \in \mathcal {A}_k$, by definition one can find $f_{k+1}, \dots, f_{N-1}$ such that $f_kx^kt^x + f_{k+1}x^{k+1}t^x + \cdots + f_{N-1}x^{N-1}t^x \in j_{!+} I_X^{0, N}$. Since $j_{!+} I_X^{0, N} \subseteq \bigoplus_l F_l \O_\X(^\dag X_s)_\Q x^l t^x$ by (i), we get $f_{l} \in F_l \O_\X(^\dag X_s)_\Q$ for $l > k$. By induction, $f_lx^lt^x \in j_{!+} I_X^{0, N}$. Thus we conclude that $f_k x^k t^x \in j_{!+} I_X^{0, N}$. In particular $t^x \in j_{!+} I_X^{0,N}$ since $1 \in \mathcal A_{0} = F_0 \O_\X(^\dag X_s)_\Q = \O_{\X, \Q}$.
    \end{proof}
    
    \begin{proposition}
        The (unipotent) nearby cycle $\Psi := \Psi \O_{X_\eta}$ is isomorphic to $j_+ I_X^{0, N} / j_{!+} I_X^{0, N}(\delta_X) \in {\rm Hol}(X, X, \P)$ for $N \geq \delta_X.$
    \end{proposition}
    \begin{proof}
        Let us fix a data of lifts $((\P_\alpha), (\X_\alpha), u)$. Under the categorical equivalence of Theorem \ref{u_0^*},
        the kernel and cokernel of $j_! I^{a, b}_X \rightarrow j_+ I^{a, b}_X$ correspond to the kernel and cokernel of $u_\alpha^! (j_! I^{a, b}_X|_{\P_\alpha}) \rightarrow u_\alpha^! (j_+ I^{a, b}_X|_{\P_\alpha})$. By Proposition \ref{ie_form} and Remark \ref{cons_trun_form}, the cokernel is killed by $x^{\delta_X}$ since $F_{\delta_X} \O_{\X_\alpha}(^\dag X_{\alpha, s})_\Q = \O_{\X_\alpha}(^\dag X_{\alpha, s})_\Q$. The kernel is also killed by $x^{\delta_X},$ since in general $\DD \ker(j_! \E \rightarrow  j_+ \E) \simeq \coker(j_! \DD \E \rightarrow j_+ \DD \E)$ and we have $\DD(x) = x$ by (\ref{dual-x}).
        This implies that $\displaystyle \lim_{\longleftrightarrow} j_! I_X^{\cdot, \cdot} \xrightarrow{\sim} \lim_{\longleftrightarrow} j_+ I_X^{\cdot, \cdot}$ is an isomorphism by \cite[3.17]{Lich}. Let $\O_{X_\eta}^{a, b} := \O_{X_\eta} \otimes I_X^{a, b}$. Since $\O_{X_\eta}^{a, b} \simeq \DD_{X_\eta}(\O_{X_\eta})(\delta_X)$ by \cite[3.11]{frob} and $\DD_{X_\eta}(\O_{X_\eta}) \otimes \E \simeq \E$, we have $\O_{X_\eta}^{a, b} \simeq I_X^{a, b}(\delta_X).$ We get
        \begin{eqnarray*}
            \Psi \O_{X_\eta} & = & \coker(\lim_{\longleftrightarrow}j_!\O_{X_\eta, 0}^{\cdot, \cdot} \rightarrow \lim_{\longleftrightarrow}j_+\O_{X_\eta, 0}^{\cdot, \cdot})\\
            & \simeq & \varprojlim_b \coker (j_!I_X^{0,b} \rightarrow j_+ I_X^{0,b})(\delta_X)\\
            & \simeq & j_+ I_X^{0,N} / j_{!+} I_X^{0,N}(\delta_X).
        \end{eqnarray*}
        The middle isomorphism comes from the proof of \cite[3.21]{Lich}.
    \end{proof}
    
    \begin{proposition}
        There is a distinguished triangle $\Psi(1) \xrightarrow{x} \Psi \rightarrow i^+j_+ \O_{X_\eta}(1) \rightarrow.$
    \end{proposition}
    \begin{proof}
        Let $N$ be a sufficiently large integer and let $\Xi = \coker (j_! I^{1, N+1} \rightarrow j_+ I^{0,N+1})(\delta_X).$ By the commutative diagram with exact rows
        \[
        \xymatrix {
            0 \ar[r] & j_! I_X^{1,N+1}(\delta_X) \ar@{=}[r] \ar[d] & j_! I_X^{1, N+1}(\delta_X) \ar[r]\ar[d] & 0 \ar[r] \ar[d] & 0 \\
            0 \ar[r] & j_+ I_X^{1,N+1}(\delta_X) \ar[r]             & j_+ I_X^{0, N+1}(\delta_X) \ar[r] & j_+I_X^{0,1}(\delta_X) \ar[r] & 0,
        }
        \]
        the snake lemma induces the exact sequence
        $$0 \rightarrow \Psi(1) \xrightarrow{\beta_+} \Xi \xrightarrow{\alpha_+} j_+ I_X^{0,1}(\delta_X) \rightarrow 0.$$
        By the commutative diagram with exact rows:
        \[
        \xymatrix{
            0 \ar[r] & j_! I_X^{1,N+1}(\delta_X) \ar[r] \ar[d] & j_! I_X^{0, N+1}(\delta_X) \ar[r]\ar[d] & j_! I_X^{0,1}(\delta_X) \ar[r] \ar[d] & 0 \\
            0 \ar[r] & j_+ I_X^{0,N+1}(\delta_X) \ar@{=}[r]    & j_+ I_X^{0, N+1}(\delta_X) \ar[r] & 0 \ar[r] & 0,
        }
        \]
        the snake lemma induces an exact sequence
        $$\ker(j_! I_X^{0,N+1} \rightarrow j_+I_X^{0,N+1}(\delta_X)) \xrightarrow{a} j_!I_X^{0,1}(\delta_X) \rightarrow \Xi \rightarrow \Psi \rightarrow 0.$$
        The dual of $a$ is the canonical morphism $j_+ I_X^{-1, 0} \rightarrow \coker(j_!I_X^{-N-1, 0} \rightarrow j_+I_X^{-N-1, 0})$, which is the zero map for $N$ large enough, by Proposition \ref{ie_form}. Therefore we get an exact sequence
        $$0 \rightarrow j_!I_X^{0,1}(\delta_X) \xrightarrow{\alpha_-} \Xi \xrightarrow{\beta_-} \Psi \rightarrow 0.$$
        The composition $\alpha_+ \alpha_-$ is equal to the canonical homomorphism $\gamma_- : j_! I_X^{0,1}(\delta_X) \rightarrow j_+ I_X^{0,1}(\delta_X)$ and $\beta_-\beta_+ : \Psi(1) \rightarrow \Psi$ is equal to the multiplication by $x$.
        We have an exact sequence of complexes of holonomic modules
        \[
        \xymatrixcolsep{15mm}
        \xymatrix{
        & 0\ar[d] & 0\ar[d] & 0\ar[d] &\\
        0 \ar[r] & 0 \ar[r] \ar[d] & \Xi \ar[r]^{\alpha_+} \ar[d] & j_+I_X^{0,1}(\delta_X) \ar[r] \ar@{=}[d] & 0\\
        0 \ar[r] & j_!I_X^{0,1}(\delta_X) \ar[r]^{(\alpha_-, \gamma_-)} \ar@{=}[d] & \Xi \oplus j_+I_X^{0,1}(\delta_X) \ar[r]^{(\alpha_+, -1)} \ar[d] & j_+I_X^{0,1}(\delta_X) \ar[r] \ar[d] & 0\\
        0 \ar[r] & j_!I_X^{0,1}(\delta_X) \ar[r]^{\gamma_-}\ar[d] & j_+I_X^{0,1}(\delta_X) \ar[r]\ar[d] & 0 \ar[r]\ar[d] & 0\\
        & 0 & 0 & 0. &
        }
        \]
        This induces a distinguished triangle
        $$ \Psi(1) \xrightarrow{x} \Psi \rightarrow i^+j_+I_X^{0,1}(\delta_X) \rightarrow.$$
        Since $I_X^{0,1} \simeq \O_{X_\eta}(1-\delta_X), $ we get the assertion.
        
    \end{proof}
    
    Next, let us calculate $\DD_\X(\Psi).$ We have a commutative diagram of $F$-overholonomic modules on $\X$ with exact rows:
    \[
    \xymatrix{
        0 \ar[r] & j_!I_X^{0, N} \ar[r]\ar[d] & j_!I_X^{-N, N} \ar[r]\ar[d] & j_!I_X^{-N, 0} \ar[r]\ar[d] & 0\\
        0 \ar[r] & j_+I_X^{0, N} \ar[r]       & j_+I_X^{-N, N} \ar[r]       & j_+I_X^{-N, 0} \ar[r]       & 0
    }
    \]
    By the snake lemma, this induces an exact sequence
    $$\ker(j_!I_X^{-N, N} \rightarrow j_+I_X^{-N, N}) \xrightarrow{a} \ker(j_!I_X^{-N, 0} \rightarrow j_+I_X^{-N, 0}) \xrightarrow{b} \coker(j_!I_X^{0, N} \rightarrow j_+I_X^{0, N}) \xrightarrow{c} \coker(j_!I_X^{-N, N} \rightarrow j_+I_X^{-N, N}).$$
    $c$ is a zero map for $N$ sufficiently large by a variant of Proposition \ref{ie_form}. Since $a$ is the dual of $c$, $a$ is also the zero map; we conclude that $b$ is an isomorphism.
    \begin{eqnarray*}
        \ker(j_!I_X^{-N, 0} \rightarrow j_+I_X^{-N, 0}) & \simeq & \ker(j_!\DD I_X^{0, N} \rightarrow j_+ \DD I_X^{0, N}) (\delta_X-1)\\
        & \simeq & \ker(\DD j_+I_X^{0, N} \rightarrow \DD j_!I_X^{0, N}) (1-\delta_X)\\
        & \simeq & \DD(\coker(j_!I_X^{0, N} \rightarrow j_+I_X^{0, N})) (1-\delta_X)\\
        & \simeq & \DD(\Psi)(1).
    \end{eqnarray*}
    Thus $b$ induces an isomorphism $b: \DD(\Psi)(1) \xrightarrow{\sim} \Psi(-\delta_X).$ Moreover the commutative diagram
    \[
    \xymatrixcolsep{15mm}
    \xymatrix{
        \DD I_{X}^{a, b}(1) \ar[r]^{\DD x}\ar[d]^{\wr} & \DD I_{X}^{a, b} \ar[d]^{\wr}\\
        I_{X}^{-b, -a}(-\delta_X) \ar[r]^{x} & I_{X}^{-b, -a}(-\delta_X-1)
    }
    \]
    induces a commutative diagram
    \[
    \xymatrixcolsep{15mm}
    \xymatrix{
        \DD (\Psi)(1) \ar[r]^{\DD x}\ar[d]^{\wr} & \DD\Psi \ar[d]^{\wr}\\
        \Psi(-\delta_X) \ar[r]^{x} & \Psi(-\delta_X-1).
    }
    \]
    Let us summarize the results of this subsection.
    \begin{theorem} \label{nc_form}
        The nearby cycle $\Psi$ is $j_+ I_X^{0,N} / j_{!+} I_X^{0,N}(\delta_X)$ for $N \geq \delta_X.$ Locally if there is a lift $\X \rightarrow \DD_\S$ of $X \rightarrow \DD_S$ with $\X$ locally having finite $p$-basis, $\Psi$ is represented by the $\D^\dag_\X(^\dag X_s)_\Q$-module $\bigoplus_{k \geq 0} \O_\X(^\dag X_s)_\Q / F_k \O_\X(^\dag X_s)_\Q\, x^k t^x.$
        
        There is a distingusihed triangle $ \Psi(1) \xrightarrow{x} \Psi \rightarrow i^+j_+ \O_{X_\eta}(1) \rightarrow.$
        
        Also there is a canonical isomorphism $b : \DD(\Psi)(1) \xrightarrow{\sim} \Psi(-\delta_X)$ which fits into the commutative diagram
    \[
    \xymatrixcolsep{15mm}
    \xymatrix{
        \DD (\Psi)(1) \ar[r]^{\DD x}\ar[d]^\wr & \DD\Psi \ar[d]^\wr\\
        \Psi(-\delta_X) \ar[r]^x & \Psi(-\delta_X-1).
    }
    \]
    \end{theorem}
    
    \subsection{Filtrations on $\Psi$}
    The canonical map $x : I^{a, b}_X \rightarrow I^{a, b}_X(-1)$ induces an endomorphism $x : \Psi \rightarrow \Psi(-1)$. $x^{\delta_X} = 0$ by Theorem \ref{nc_form} and we can define the kernel, image and monodromy filtrations on $\Psi$.
    
    \begin{definition}
        We define the kernel and image filtrations on $\Psi$:
        $$F_p \Psi := \ker (x^{p+1} : \Psi \rightarrow \Psi),$$
        $$G^q \Psi := \im (x^q : \Psi \rightarrow \Psi).$$
        The monodromy filtration on $\Psi$ is defined by:
        $$M_r \Psi := \sum_{p-q = r} (F_p \Psi \cap G_q \Psi).$$ 
        We also define follwing filtrations:
        $$F_p Gr_G^q \Psi := \im (F_p \Psi \cap G^q \Psi \rightarrow Gr_G^q \Psi),$$
        $$G^q Gr^F_p \Psi := \im (G^q \Psi \cap F_p \Psi \rightarrow Gr^F_p \Psi).$$
    \end{definition}
    
    If (locally) $X$ lifts to $\X$, then we have
    $$F_p \Psi \simeq \bigoplus_{k \geq 0} F_{k+p+1} \O_{\X}(^\dag X_s)_\Q/ F_{k} \O_{\X}(^\dag X_s)_\Q\, x^k t^x,$$
    $$G^q \Psi \simeq \bigoplus_{k \geq q} \O_{\X}(^\dag X_s)_\Q/ F_{k} \O_{\X}(^\dag X_s)_\Q\, x^k t^x.$$
    Thus we have 
    $$Gr^F_p Gr_G^q \Psi = Gr_{p+q+1} \O_{\X}(^\dag X_s)_\Q(1) \, x^q t^x \simeq \bigoplus_{\card I = p+q+1} \RG_{D_I}(\O_{\X, \Q})(q+1)[p+q+1].$$
    These isomorphisms glue via the equivalence Theorem \ref{u_0^*}:
    
    \begin{theorem} \label{nc_gr}
        Let $c = \delta_P - \delta_X$.
        $$\displaystyle Gr^F_p Gr_G^q \Psi \simeq \bigoplus_{\card I = p+q+1} \RG_{D_I}(\O_{\P, \Q})(q+1)[c+p+q+1].$$
        As a consequence, $$\displaystyle Gr^M_r \Psi \simeq \bigoplus_{\substack{p-q=r, \\ p,q \geq 0}} \bigoplus_{\card I = p+q+1} \RG_{D_I}(\O_{\P, \Q})(q+1)[c+p+q+1].$$
    \end{theorem}
    \begin{proof}
        Fix a data of lifts $(X_\alpha, \P_\alpha, u)$. We have
        $$u_0^* \Psi = (u_\alpha^!(\Psi|\P_\alpha), \star) \simeq (\bigoplus_{k\geq0} F_k \O_{\X_\alpha}(^\dag X_{\alpha, s})_\Q \, x^kt^x(\delta_X), \star).$$
        $\star$ is the canonical isomorphisms defined by $\tau$ (see the definition of $u_0^*$) and the second isomorphism is induced by Proposition \ref{ie_form}.
        Since $u_0^*$ is an equivalence of abelican categories, we have a canonical isomorphism
        $$u_0^* Gr^F_p Gr_G^q \Psi \simeq Gr^F_p Gr_G^q u_0^* \Psi \simeq (Gr_{p+q+1} \O_{\X_\alpha}(^\dag X_{\alpha, s})_\Q (1) \, x^q t^k, \star) \simeq (Gr_{p+q+1} \O_{\X_\alpha}(^\dag X_{\alpha, s})_\Q (q+1), \star).$$
        By Remark \ref{cons_trun_form}, we have $(Gr_{p+q+1} \O_{\X_\alpha}(^\dag X_{\alpha, s})_\Q (q+1), \star) \simeq u_0^* Gr_{p+q+1} \O_{X_\eta}(q+1).$ Finally Theorem \ref{cons_trun} implies the assertion.
    \end{proof}
    
    Imitating the $l$-adic case, we want the ``correct" nearby cycle $R\psi\one_{X_\eta}$ which fits into the distinguished trianle $$R\psi\one_{X_\eta} \rightarrow i^+j_+ \one_{X_\eta}(1)[1] \rightarrow R\psi\one_{X_\eta}(1)[1] \rightarrow .$$
    As $\one_{X_\eta} \simeq \O_{X_\eta}(-\delta_X)[-\delta_X]$, it is natural to make the following definition:
    
    \begin{definition}
        $R\psi\one_{X_\eta} := \Psi(-\delta_X)[1-\delta_X].$
    \end{definition}
    
    By definition, $R\psi\one_{X_\eta}$ is a $(1-\delta_X)$-shifted holonomic $\D$-module. It admits filtrations induced by those of $\Psi.$ By Theorem \ref{nc_gr} and (\ref{one_calc}), we calculate
    \begin{align} \label{rpsi_grgr}
        Gr^F_pGr_G^q R\psi\one_{X_\eta} & \simeq \bigoplus_{\card I = p+q+1} \RG_{D_I}(\O_{\P, \Q})(-\delta_X + q + 1)[c+p+q+2-\delta_X] \\
        & \simeq \bigoplus_{\card I = p+q+1} \one_{D_I}(-p)[-(p+q)]. \notag
    \end{align}
    \begin{align} \label{rpsi_gr}
        Gr^M_r R\psi\one_{X_\eta} & \simeq \bigoplus_{\substack{p-q=r, \\ p,q \geq 0}} \bigoplus_{\card I = p+q+1} \RG_{D_I}(\O_{\P, \Q})(-\delta_X + q + 1)[c+p+q+2-\delta_X] \\
        & \simeq \bigoplus_{\substack{p-q=r, \\ p,q \geq 0}} \bigoplus_{\card I = p+q+1} \one_{D_I}(-p)[-(p+q)]. \notag
    \end{align}

    The rest of this subsection is dedicated to show that $F_p \Psi$ is a constructible truncation of $\Psi$.
    
    \begin{theorem} \label{cons_trun_nc}
        $Gr^F_p \Psi \in D^{b, -\delta_X +p+1}_{\rm cons}$. As a consequence, 
        $F_p \Psi \simeq \tau^{\leq -\delta_X+p+1}_{\rm cons} \Psi.$
    \end{theorem}
    Since $u_0^*$ preserves constructible t-structure and dual constuctible t-structure by Proposition \ref{u_0^*-cons-t} and Proposition \ref{u_0^*-dual}, we consider locally and fix a lift $\X$ of $X$. By the characterization of kernel and image filtrations, $G^q\DD(\Psi) \simeq \DD(\Psi/F_{q-1}\Psi) $ and hence $\DD (Gr^F_p\Psi) \simeq Gr_G^p(\DD (\Psi))$. By Theorem \ref{nc_form}, $Gr_G^p(\DD (\Psi)) \simeq Gr_G^p(\Psi) \simeq \O_\X(^\dag X_s)_\Q / F_p \O_\X(^\dag X_s)_\Q.$ Thus
    \begin{align} \label{dgrfpsi}
        \DD(Gr^F_p \Psi) \simeq \O_\X(^\dag X_s)_\Q / F_p \O_\X(^\dag X_s)_\Q.
    \end{align}
    
    We are going to show that $\O_\X(^\dag X_s)_\Q / F_p \O_\X(^\dag X_s)_\Q \in D^{b, \delta_X -p-1}_{\rm dcons}(X).$
    
    For $I \subseteq \{0, \dots, r\}$, let $\tilde{\one}_{D_I} := \DD(\one_{D_I}) = \RG_{D_I} \O_{\X, \Q}[\delta_X]$ and $\tilde{\one}_{X_s} := \DD(\one_{X_s}) = \RG_{X_s} \O_{\X, \Q}[\delta_X].$
    \\
    Note that $\tilde{\one}_{X_s}$ is the unit of $\tilde{\otimes} := \tilde{\otimes}_{(X_s, X_s, \X)}$. Since $\O_{\X, \Q}[-\delta_X]$ is a constructible sheaf by Lemma \ref{cons_isoc}, its dual $\O_{\X, \Q}[\delta_X]$ is a dual constructible sheaf; as $\RG_Z$ preserves dual constructible sheaves, both $\tilde{\one}_{X_s}$ and $\tilde{\one}_{D_I}$'s are dual constructible sheaves.
    \\
    Since $\tilde{\one}_{D_I} \tilde\otimes \tilde{\one}_{D_J} \simeq \tilde{\one}_{D_{I \cup J}}$, we may regard $\displaystyle \bigoplus_{\card I = p} \tilde\one_{D_I}$ as $\tilde{\bigwedge}^p (\tilde{\one}_{D_0} \oplus \cdots \oplus \tilde{\one}_{D_r}).$
    
    We have a canonical morphism $\tilde{\one}_{D_0} \oplus \cdots \oplus \tilde{\one}_{D_r} \xrightarrow{d} \tilde{\one}_{X_s} $ and can form a complex of dual constructible sheaves:
        \begin{align} \label{koszul}
            K^\cdot : \cdots \xrightarrow{d} \bigoplus_{\card I = 2} \tilde{\one}_{D_I} \xrightarrow{d} \bigoplus_{\card I = 1} \tilde{\one}_{D_I}
            \xrightarrow{d} \tilde{\one}_{X_s} \rightarrow 0.
        \end{align}
    \\
    Informally, this sequence can be regarded as a Koszul complex and $d$ is given by $$d(e_{i_0} \wedge \cdots \wedge e_{i_p}) = \sum_k (-1)^k e_{i_0} \wedge \cdots \wedge \hat{e}_{i_k} \wedge \cdots \wedge e_{i_p}.$$
    \begin{proposition} \label{koszul-exact}
        The sequence {\rm (\ref{koszul})} is exact in $D^{b, \heartsuit}_{\rm hol, dcons}(X)$.
    \end{proposition}
    \begin{proof}
        Consider the $p$-smooth stratification $\displaystyle X_p := \bigcup_{\card I = p} D_I.$ Since $\displaystyle X_p - X_{p+1} = \coprod_{\card I = p}(D_I -  X_{p+1})$ and by using Lemma \ref{dev_exact} repeatedly, we are reduced to showing that for each $I \subseteq \{0, \dots, r\}$ with $\card I = p$,  $\RG_{D_I - X_{p+1}} K^\cdot$ is an exact sequence of $D^{b, \heartsuit}_{\rm hol, dcons}(D_I - X_{p+1})$.
        \\
        Let $\U := \X - X_{p+1}.$ For $J \subseteq \{0, \dots, r\}$, we have $(\RG_{D_I - X_{p+1}} \RG_{D_J}(\O_{\X, \Q}))|_\U \simeq \RG_{D_{I\cup J}\cap\U}(\O_{\U, \Q})$.
        If $J \not\subseteq I$, then $D_{I \cup J} \cap \U = \emptyset.$ Thus $(\RG_{D_I - X_{p+1}}K^\cdot)|_\U$ is the sequence
        $$\cdots \xrightarrow{d} \bigoplus_{\card J = 2, J \subseteq I} \RG_{D_I \cap U} \O_{\U, \Q}[\delta_X] \xrightarrow{d} \bigoplus_{\card J = 1, J \subseteq I} \RG_{D_I \cap U} \O_{\U, \Q}[\delta_X] \xrightarrow{d} \RG_{D_I\cap U} \O_{\U, \Q}[\delta_X] \rightarrow 0.$$
        This is an exact sequence of $D^{b, -\delta_X + p}_{\rm hol, isoc}(D_I\cap U, \U)$ (it is indeed a Koszul complex in this situation) and therefore $\RG_{D_I - X_{p+1}} K^\cdot$ is an exact sequence of $D^{b, \heartsuit}_{\rm hol, dcons}(D_I\cap U, \U).$
    \end{proof}
    
    We have a distinguished triangle induced by the exact sequence of holonomic modules
    $$ Gr^F_p \O_\X(^\dag X_s)_\Q \rightarrow \O_\X(^\dag X_s)_\Q / F_{p-1}\O_\X(^\dag X_s)_\Q \rightarrow \O_\X(^\dag X_s)_\Q / F_p\O_\X(^\dag X_s)_\Q \rightarrow .$$
    For simplicity, we will write $F_\infty = \O_\X(^\dag X_s)_\Q, F_p = F_p \O_\X(^\dag X_s)_\Q$ and $Gr_p = Gr^F_p.$
    Let $\alpha_p : F_\infty/F_p[-1] \rightarrow Gr_p$ and $\beta_p : Gr_p \rightarrow F_\infty/F_{p-1}$ be the canonical morphisms in the derived category.
    
    \begin{proposition} \label{alphabeta}
        The diagram
        \[
        \xymatrixcolsep{20mm}
        \xymatrix{
                    Gr_{p+1}[-1] \ar[r]^{-\alpha_p \circ (\beta_{p+1}{[-1]})} & Gr_p\\
                    \displaystyle\bigoplus_{\card I = p+1} \RG_{D_I}(\O_{\X, \Q})[p] \ar[r]^{d[-\delta_X+p]} \ar[u]^{\wr}_{\rm \ref{cons_trun}} & \displaystyle\bigoplus_{\card I = p} \RG_{D_I}(\O_{\X, \Q})[p] \ar[u]^{\rm \ref{cons_trun}}_{\wr}
        }
        \]
        is commutative.
    \end{proposition}
    \begin{proof}
        Let $I = \{ i_1 < \cdots < i_p\} \subseteq \{0, \dots, r\}$. Consider the following morphism
        $$
            \iota_I : \mathcal{C}^\dag_I := [\O_{\X, \Q} \rightarrow \bigoplus_{j_1} \O_\X(^\dag D_{i_{j_1}})_\Q \rightarrow \cdots \rightarrow \O_\X(^\dag D_{i_1} \cup \cdots \cup D_{i_p})_\Q] \rightarrow F_p/F_{p-1},
        $$
        where the homomorphism at the degree 0 is $\O_\X(^\dag D_{i_1} \cup \cdots \cup D_{i_p})_\Q \hookrightarrow F_p \rightarrow F_p/F_{p-1}.$ Via the canonical isomorphism $\RG_{D_I}(\O_{\X, \Q})[p] \simeq \mathcal C^\dag_I$ of Lemma \ref{sp_resol}, the vertical isomorphisms of the diagram in the statement of this proposition are $\oplus \iota_I.$
        For $I' \subseteq I$ with $\card I' = p'$ and $\card I = p$, there is a canonical morphism $c_{II'} : \mathcal C^\dag_I[-p] \rightarrow \mathcal C^\dag_{I'}[-p'],$ which is simply the projection at each degree.
        The morphism $d' := \displaystyle d[-\delta_X+p] : \displaystyle\bigoplus_{\card I = p+1} \mathcal{C}^\dag_I[-1] \rightarrow \displaystyle\bigoplus_{\card I = p} \mathcal{C}^\dag_I$ is defined as follows: for $I = \{i_0 < \cdots < i_p \},$ $d'$ on its $I$-th direct summand is $d'_I = \sum_k (-1)^k c_{I, I-\{i_k\}}[p].$
        Let us calculate the composition $\alpha_p\circ(\beta_{p+1}[-1]) : Gr_{p+1}[-1] \rightarrow Gr_{p}.$ We have a commutative diagram of complexes
        \begin{equation}
            \label{alpha'beta'}
            \xymatrix{
                {[F_p/F_{p-1} \hookrightarrow F_{p+1}/F_{p-1}]} \ar[d]^{\rm q\mathchar`-iso} \ar[r] & {[F_p/F_{p-1} \hookrightarrow F_\infty/F_{p-1}]} \ar[d]^{\rm q\mathchar`-iso} \ar[rd] & \\
                {[0 \rightarrow F_{p+1}/F_p]}  \ar[r]^{\beta_{p+1}{[-1]}} & {[0 \rightarrow F_\infty/F_p]} \ar@{-->}[r]^{\alpha_p} & {[F_p/F_{p-1} \rightarrow 0]},
            }
        \end{equation}
        where ``q-iso" means quasi-isomorphism. This diagram is the composition $\alpha_p(\beta_{p+1}[-1])$. 
        Let $I = \{i_0 < \cdots < i_p\}$. Define the morphism of complexes $\iota'_I : \mathcal C^\dag_I[-1] \rightarrow [F_p/F_{p-1} \rightarrow F_{p+1} / F_{p-1}]$ as follows: 
        at degree 0, ${\iota'_I}^0 : \bigoplus_{q} \O_\X(^\dag D_{i_0} \cup \cdots \cup \hat{D}_{i_q} \cup \cdots \cup D_{i_p})_\Q \rightarrow F_p \rightarrow F_p / F_{p-1}$ is given by $(f_0, \dots, f_p) \mapsto \sum_q (-1)^{q+1} f_q$; 
        at degree 1, ${\iota'_I}^1$ is simply $\O_\X(^\dag D_{i_0} \cup \cdots \cup D_{i_p})_\Q \hookrightarrow F_{p+1} \rightarrow F_{p+1}/F_{p-1}.$
        It is easy to see that $\iota'_I$ is a morphism of complexes and the diagram below is commutative:
        \begin{equation} \label{iota'}
        \tag{$\star$}
        \xymatrix {
            \displaystyle\bigoplus_{\card I = p+1} \mathcal C^\dag_I[-1] \ar[rr]^{\oplus \iota_I[-1]}_{\rm q\mathchar`-iso} \ar[rd]^{\oplus \iota'_I} &  & F_{p+1}/F_p[-1]\\
            & {[F_p/F_{p-1} \hookrightarrow F_{p+1}/F_p],} \ar[ur]^{\rm q\mathchar`-iso} &
        }
        \end{equation}
        where the right arrow is the left vertical arrow of (\ref{alpha'beta'}). It follows that $\oplus \iota'_I$ is a quasi-isomorphism.
        
        Consider the following diagram of complexes:
        \[\xymatrix{
            {[F_p/F_{p-1} \hookrightarrow F_{p+1}/F_{p-1}]} \ar[dr]^{\rm q\mathchar`-iso} \ar[rr] & & {[F_p/F_{p-1} \hookrightarrow F_\infty/F_{p-1}]} \ar[dr] &\\
            & F_{p+1}/F_p[-1] \ar@{-->}[rr]^{\alpha_p(\beta_{p+1}{[-1]})} & & F_p/F_{p-1}\\
            \displaystyle\bigoplus_{\card I = p+1} \mathcal C^\dag_I[-1] \ar[ur]^{\oplus \iota_I{[-1]}} \ar[rr]^{d'} \ar[uu]^{\oplus \iota'_I}_{\rm q\mathchar`-iso} & & \displaystyle\bigoplus_{\card I = p} \mathcal C^\dag_I. \ar[ur]^{\oplus \iota_I} &
        }\]
        The left triangle is (\ref{iota'}) and the top square is (\ref{alpha'beta'}).
        The dashed arrow is a morphism in the derived category, and the assertion of this proposition is that the bottom square is commutative up to the multiplication by $-1$. As $\oplus \iota'_I$ is a quasi-isomorphism, it suffices to show that the outer pentagon is commutative in the category of complexes. We only have to consider the maps at degree 0. The composition of left and top two maps is
        $$\displaystyle \bigoplus_{i_0 < \cdots < i_p} {\iota'_I}^0 : \bigoplus_{i_0 < \cdots < i_p} \bigoplus_q \O_\X(^\dag D_{i_0} \cup \cdots \cup \hat{D}_{i_q} \cup \cdots \cup D_{i_p})_\Q \rightarrow F_p/F_{p-1}.$$
        The composition of bottom two maps is
        $$\bigoplus_{i_0 < \cdots < i_p} \bigoplus_q \O_\X(^\dag D_{i_0} \cup \cdots \cup \hat{D}_{i_q} \cup \cdots \cup D_{i_p})_\Q \xrightarrow{d'^0} \bigoplus_{i_1 < \cdots < i_p}\O_\X(^\dag D_{i_1} \cup \cdots \cup D_{i_p})_\Q \xrightarrow{\oplus \iota_I^0} F_p/F_{p-1}.$$
        Fix an $I = {i_0 < \cdots < i_p}$ and take $(f_0, \dots, f_p) \in \bigoplus_q \O_\X(^\dag D_{i_0} \cup \cdots \cup \hat{D}_{i_q} \cup \cdots \cup D_{i_p})_\Q.$ We have
        $${\iota'_I}^0(f_0, \dots, f_p) = \sum_q (-1)^{q+1} f_q.$$
        On the other hand,
        $$(\oplus \iota_{I'})d'^0(f_0, \dots, f_p) = (\oplus \iota_{I'})(\sum_q (-1)^q f_q) = \sum_q (-1)^q f_q,$$
        where in the middle term, $f_q$ is regarded as an element of $\mathcal C^{\dag 0}_{I-\{i_q\}}.$ So we have proved the proposition.
    \end{proof}
    
    \noindent
    (Proof of Theorem \ref{cons_trun_nc})
    
    We keep the symbols of Proposition \ref{alphabeta}. Put $K_p := Gr_p[\delta_X - p], \ I_p := F_\infty/F_p[\delta_X-p-1], \ \alpha'_p := \alpha_p[\delta_X-p]:I_p \rightarrow K_p, \ \beta'_p := \beta_p[\delta_X-p]: K_p \rightarrow I_{p-1}.$ Let $d'_p : K_p \rightarrow K_{p-1}$ be the morphism induced by the isomorphism of Remark \ref{cons_trun_form} and (\ref{koszul}). We have to show that each $I_p$ is a dual constructible sheaf.
    
    Proposition \ref{koszul-exact} asserts that $\cdots \rightarrow K_{p+1} \xrightarrow{d'_{p+1}} K_p \xrightarrow{d'_p} K_{p-1} \rightarrow \cdots$ is an exact sequence in $D^{b, \heartsuit}_{\rm hol, dcons}$ Alos, by Proposition \ref{alphabeta}, $\alpha'_{p-1}\beta'_p = -d'_p.$ By definition, there is a distinguished triangle in $D^{b}_{\rm hol}$:
    \begin{equation} \label{alphabeta_t}
        I_p \xrightarrow{\alpha'_p} K_p \xrightarrow{\beta'_p} I_{p-1} \rightarrow. \tag{$\ast$}
    \end{equation}
    
    Let us prove that $I_p \in D^{b, \heartsuit}_{\rm hol, dcons}$, $\alpha'_p$ is injective and $\im(\alpha'_p) = \ker(d'_p)$ by descending induction on $p$. It is obvious that it is true for $p$ sufficiently large. Suppose that the assertion is true for $p$. Then by the distinguished triangle (\ref{alphabeta_t}), we conclude that $I_{p-1} \in D^{b, \heartsuit}_{\rm hol, dcons}$ and that $\beta'_p$ is surjective. Since $\alpha'_{p-1}\beta'_{p} = -d'_p,$ we have $\ker(d'_{p-1}) = \im(d'_p) = \im(\alpha'_{p-1}).$ Let us show that $\alpha'_{p-1}$ is injective by diagram chasing: for $x \in \ker(\alpha'_{p-1}),$ choose some $y \in K_p$ such that $\beta'_p(y) = x.$ Then since $d'_p(y) = -\alpha'_{p-1} \beta'_p(y) = -\alpha'_{p-1}(x) = 0$, we have $y \in \ker(d'_p) = \im(\alpha'_p).$ Choose some $z \in I_p$ such that $y = \alpha'_p(z).$ Then we have $x = \beta'_p(y) = \beta'_p \alpha'_p(z) = 0.$
    
    As a by-product, we have the following theorem:
    \begin{corollary}
        (i) We have an exact sequence of constructible sheaves
        $$ \displaystyle 0 \rightarrow \one_{X_s} \xrightarrow{\delta} \bigoplus_{\card I = 1} \one_{D_I} \xrightarrow{\delta\wedge} \bigoplus_{\card I = 2} \one_{D_I} \rightarrow \cdots.$$
        \\
        (ii) There exist short exact sequences of constructible sheaves
        $$0 \rightarrow \H^p_{\rm cons}(R\psi\one_{X_\eta}) \xrightarrow{\alpha_p} i^+\H^{p+1}_{\rm cons}(j_+ 
\one_{X_\eta})(1) \xrightarrow{\beta_{p+1}} \H^{p+1}_{\rm cons}(R\psi\one_{X_\eta})(1) \rightarrow 0$$
        such that $\alpha_p \beta_p$ agrees with $\delta \wedge$ via the isomorphism $i^+\H^{p+1}_{\rm cons}(j_+\one_{X_\eta})(p+1) \simeq \bigoplus_{\card I = p+1} \one_{D_I}$ of {\rm Proposition \ref{purity}}.
    \end{corollary}
    \begin{proof}
        (i) Take the dual of (\ref{koszul}). (ii) Using the notations of (\ref{alphabeta_t}), we have proved that the sequence
        $$ 0 \rightarrow I_p \xrightarrow{\alpha'_p} K_p \xrightarrow{\beta'_p} I_{p-1} \rightarrow 0$$
        is exact in $D^\heartsuit_{\rm hol, dcons}$. By (\ref{dgrfpsi}) and Theorem \ref{cons_trun_nc}, the dual of $I_p$ is $\H^{p+1}_{\rm cons}(R\psi\one_{X_\eta})$. $K_p$ is canonically isomorphic to $\bigoplus_{\card I = p} \tilde{\one}_{D_I}$ by Theorem \ref{cons_trun_form}. Since $\alpha'_{p-1}\beta'_p = -d'_p$, the dual of this exact sequence with suitable signs is what (ii) asserts.
    \end{proof}
    
    \begin{remark}
        In $l$-adic case, one can construct weight spectral sequence from this corollary; choose a complex $K$ and a morphism $K \rightarrow K(1)[1]$ representing $i^*Rj_* \Q_l$ and the cup product $i^*Rj_* \Q_l \rightarrow i^*Rj_* \Q_l(1)[1].$ Then by this corollary, $R \psi \Q_l$ is isomorphic to the total complex of $A := [(\tau^{\geq 1}K)[1] \rightarrow (\tau^{\geq 2}K)[2] \rightarrow \cdots].$ Endow $A$ with the filtration $P_kA := [(\tau^{[1, k+1]}K)[1] \rightarrow (\tau^{[2, k+3]}K)[2] \rightarrow \cdots]$; this filtration induces the weight spectral sequence. In $l$-adic case, this is equivalent to the construction using perverse sheaves and monodromy filtrations. For more detail, see \cite{l-adic}.
        
        It is not easy to imitate this construction in our situation since we do not have an explicit morphism of complexes $K \rightarrow K(1)[1]$ representing $i^+j_+\one_{X_\eta} \rightarrow i^+j_+\one_{X_\eta}(1)[1]$. In the spirit of Riemann-Hilbert correspondence, perverse t-structure of $l$-adic coefficients ``corresponds" holonomic t-structure of arithmetic $\D$-modules and it is much easier to treat the monodromy filtration.
    \end{remark}
    
    \section{Weight spectral sequence} \label{sec-wss}
    \subsection{Construction}
    Let $X, \P$ and $D_0, \dots, D_r$ be the same as the previous section.
    Let $\pi : (X, X, \P) \rightarrow (\DD_S, \DD_S, \DD_\S)$, $\pi_s : (X_s, X_s, \P_s) \rightarrow (s, S, \S)$ and $\pi_\eta : (X_\eta, X, \P) \rightarrow (\eta, \DD_S, \DD_\S)$ be the structure morphisms. Let $\E^\dag_K = \Gamma(\DD_\S, \O_{\DD_\S}(^\dag s)_\Q) = \{\sum_i a_i t^i \in K[[t, t^{-1}]]: \sup_i |a_i| < \infty, \ \exists \eta < 1 (|a_i|\eta^i \rightarrow 0 \ {\rm as} \ i \rightarrow -\infty)\}$ be the bounded Robba ring over $K$. Note that this is a field which is contained in the Amice ring $\E_K := \V[[t]] \langle t^{-1}\rangle_\Q$ which is also a field.
    
    \begin{definition} \label{ad-coh}
        For $\E \in D^b_{\rm hol}(X_s, X_s, \P_s),$ define the arithmetic $\D$-module cohomology of $\E$ by $\Rg_\D(X_s; \, \E/K) := \pi_{s+}\E$ and $H^i_{\D}(X_s; \, \E/K) := H^i(\Rg_\D(X_s; \, \E/K))$.
        
        For $\F \in D^b_{\rm hol}(X_\eta, X, \P),$ define the arithmetic $\D$-module cohomology of $\F$ by $\Rg_\D(X_\eta; \, \F/\E^\dag_K) := \Rg(\DD_\S, \pi_{\eta+}\F)$ and $H^i_{\D}(X_\eta; \, \F/\E^\dag_K) := H^i(\Rg_\D(X_\eta; \, \F/\E^\dag_K))$. When $\F = \one_{X_\eta}[1]$, we abbreviate them to $\Rg_\D(X_\eta/\E^\dag_K)$ and $H^i_\D(X_\eta/\E^\dag_K)$.
    \end{definition}
    
    $R\psi\one_{X_\eta}$ is supported on $X_s$ and we regard it as an object in $D^b_{\rm hol}(X_s, X_s, \P_s)$ by Theorem \ref{crys-phil}; since $R\psi\one_{X_\eta}$ is equipped with the monodromy filtration, by applying $\Rg_\D(-)$ we get the spectral sequence
    $$E_1^{p, q} = H_\D^{p+q}(X_s; \, Gr^M_{-p} R\psi\one_{X_\eta}/K) \Rightarrow H_\D^{p+q}(X_s; \,  R\psi\one_{X_\eta}/K).$$
    Let us calculate the $E_1$-term. By (\ref{rpsi_gr}), we have
    $$\displaystyle Gr^M_{-p}(R\psi\one_{X_\eta}) \simeq \bigoplus_{\substack{i-j=-p, \\ i,j \geq 0}} \bigoplus_{\card I = i+j+1} \one_{D_I}(-i)[-(i+j)]$$
    and
    \begin{align} \label{E1term}
    E_1^{p,q} & \simeq \bigoplus_{i-j = -p} \bigoplus_{\card I = i+j+1} H_\D^{p+q-(i+j)}(X_s; \,  \one_{D_I}(-i)/K) \\ 
    & = \bigoplus_{i \geq \max(0, -p)} \bigoplus_{\card I = p+2i+1} H^{q-2i}_\D(X_s; \,  \one_{D_I}/K)(-i) \notag \\
    & = \bigoplus_{i \geq \max(0, -p)} H^{q-2i}_\D(D^{(p+2i)}/K)(-i). \notag
    \end{align}
    where we put $\displaystyle D^{(m)} = \coprod_{\card I = m+1} D_I.$
    Fix an $I \subseteq \{0, \dots, r\}, \ \card I = p$. Let $\pi_I : (D_I, D_I, \P_s) \rightarrow (s, s, \S)$ be the structure morphism. $\one_{D_I}$ is represented by $\pi_I^+(K) \in D^b_{\rm hol}(D_I, D_I, \P_s)$ and $\Rg_\D(X_s; \, \one_{D_I})$ is isomorphic to $\pi_{I+}\pi_I^+(K).$
    By \cite{comparison}, this is isomorphic to $\Rg_{\rm rig}(D_I/K).$ Thus we have
    
    \begin{equation} \label{E1term_rig}
        E_1^{p,q} \simeq \bigoplus_{i \geq \max(0, -p)} H^{q-2i}_{\rm rig}(D^{(p+2i)}/K)(-i).
    \end{equation}
    
    Next, let us calculate the $E_\infty$-term. Since $X$ is proper, $\pi_! \simeq \pi_+$ and therefore we have
    $$\pi^n_{s+} \Psi \O_{X_\eta} \simeq \Psi \pi^n_{\eta+} \O_{X_\eta}$$
    where $\theta^n_+ := \H^n(\theta_+)$ for a morphism of frames $\theta$. This isomorphism belongs to the category $\displaystyle \lim_{\longleftrightarrow} {\rm Hol}(s, s, \S).$ Since $\pi$ is proper, the left hand side belongs to ${\rm Hol(s, s, \S)}$; hence so is the right hand side.
    \begin{align} \label{E_infty_term}
        H^n_\D(X_s; \, R\psi \one_{X_\eta}/K) & = H^n_\D(X_s; \, \Psi \O_{X_\eta}(-\delta_X)[1-\delta_X]/K) \\
        & = \pi^{n+1-\delta_X}_{s+} \Psi \O_{X_\eta}(-\delta_X) \notag \\
        & \simeq \Psi \pi^{n+1-\delta_X}_{\eta+} \O_{X_\eta}(-\delta_X) \notag \\
        & \simeq \Psi \pi^n_{\eta+}\one_{X_\eta}[1] \notag \\
        & = \Psi H^n_\D(X_\eta/\E^\dag_K). \notag
    \end{align}
    By (\ref{E1term}) and (\ref{E_infty_term}), we get the following theorem:
    \begin{theorem} \label{main-theorem}
        We have a spectral sequence of $K$-vector spaces
        $$ E^{p,q}_1 = \bigoplus_{i \geq \max(0, -p)} H^{q-2i}_{\D}(D^{(p+2i)}/K)(-i) \Rightarrow \Psi H^{p+q}_{\D}(X_\eta/\E^\dag_K).$$
    \end{theorem}
    
    Let us state some conjectures.
    \begin{conjecture} \label{ovhol_eta}
        Any $E \in {\rm Hol}(\eta, \DD_S, \DD_\S)$ is coherent as an $\O_{\DD_\S}(^\dag s)_\Q$-module.
    \end{conjecture}
    This conjecture implies that the arithmetic $\D$-module cohomology $H^i_\D(X_\eta;\, \F/\E^\dag_K)$ for $\F \in D^b_{\rm hol}(X_\eta, X, \P)$ is a finite dimensional $\E^\dag_K$-vector space; in particular $H^i_\D(X_\eta/\E^\dag_K)$ is finite dimensional.
    
    Lazda and P\'{a}l defined in \cite{LPrig} the rigid cohomology $\Rg_{\rm rig}(X_\eta/\E^\dag_K)$ over $\E^\dag_K$. It has a structure of $\nabla$-module. We expect that the methods of \cite{comparison} still works in our situation so that there is a comparison theorem between arithmetic $\D$-module cohomology over Laurent series fields and Lazda-P\'{a}l's rigid cohomology over bounded Robba rings:
    \begin{conjecture} \label{conj-comp}
        We have a canonical isomorphism
        $$\Rg_\D(X_\eta/\E^\dag_K) \simeq \Rg_{\rm rig}(X_\eta / \E^\dag_K)$$
        such that the induced isomorphism
        $$H^i_\D(X_\eta/\E^\dag_K) \simeq H^i_{\rm rig}(X_\eta / \E^\dag_K)$$
        is compatible with the structure of $\nabla$-modules and the pullbacks.
    \end{conjecture}
    \begin{remark}
        Let $p: \DD_S \rightarrow S$ be the structure morphism.
        By definition, $\one_{X_\eta} = (p_\eta \pi_\eta)^+ K \simeq \pi_\eta^+ p_\eta^+ K \simeq \pi_\eta^+ \O_{\DD_\S}(^\dag s)_\Q [-1].$ The object $\pi_\eta^+(\O_{\DD_S}(^\dag s)_\Q)$ should be the constant coefficient of rigid cohomology over $\E^\dag_K$, which explains the shift $[1]$ in the definition of arithmetic $\D$-module cohomology.
    \end{remark}
    
    The comparison theorem for rigid cohomology over $K$ is known in \cite{comparison}. Conjecture \ref{conj-comp} together with \ref{main-theorem} implies the following theorem:
    
    \begin{theorem}
        Under the Conjecture \ref{conj-comp}, we have a spectral sequence
        $$ E^{p,q}_1 = \bigoplus_{i \geq \max(0, -p)} H^{q-2i}_{\rm rig}(D^{(p+2i)}/K)(-i) \Rightarrow \Psi H^{p+q}_{\rm rig}(X_\eta/\E^\dag_K).$$
    \end{theorem}
   
    We remark that Conjecture \ref{conj-comp} is true if $X$ lifts to a formal scheme $\X$ over $\V[[t]]$ locally with finite $p$-basis over $\V$.
    
    To see this, put $\pi : \X \rightarrow \DD_\S$ to be the structure morphism; let us calculate $\pi_+ \one_{X_\eta} \simeq \pi_{s, X_s+} \O_\X(^\dag X_s)_\Q[-\delta_X].$ Here $\pi_{s, X_s+}$ is the pushforward $D^b_{\rm coh}(\D^\dag_\X(^\dag X_s)_\Q) \rightarrow D^b(\D^\dag_{\DD_\S}(^\dag s)_\Q)$ defined in \cite[7.1.4.1]{intro}. As in \cite[2.4.6.2]{Berintro}, we have a locally free resolution
    $$\D^\dag_{\DD_\S \leftarrow \X}(^\dag s, \, ^\dag X_s)_\Q \xrightarrow{\sim} (^\dag X_s)\Omega^\cdot_{\X/\DD_\S, \Q}[\delta_X-1] \otimes_{\O_\X(^\dag X_s)_\Q} \D^\dag_\X(^\dag X_s)_\Q$$
    of $\D^\dag_{\DD_\S \leftarrow \X}(^\dag s, \, ^\dag X_s)_\Q$ as complexes of $(\pi^{-1}\O_{\DD_S}(^\dag s)_\Q, \D^\dag_\X(^\dag X_s)_\Q)$-modules. This is induced by a locally free resolution
    $$ (^\dag X_s)\Omega^\cdot_{\X/\DD_\S}[\delta_X-1] \xrightarrow{\sim} (^\dag X_s)\omega_{\X/\DD_\S} = (^\dag X_s)\omega_{\X} \otimes_{\pi^{-1} \O_{\DD_\S}(^\dag s)} \pi^{-1} ((^\dag s)\omega_{\DD_\S}).$$
    Thus we obtain
    $$\pi_{\eta+} \one_{X_\eta}[1] \simeq \pi_{s, X_s+} \O_\X(^\dag X_s)_\Q[1-\delta_X] \simeq \mathbb{R}\pi_*((^\dag X_s) \Omega^\cdot_{\X/\DD_\S, \Q})$$
    as $\O_{\DD_\S}(^\dag s)_\Q$-modules. We have a sequence of isomorphisms of $E^\dag_K$-vector spaces
    \begin{align}
        \Rg(\DD_\S, \pi_{\eta+} \one_{X_\eta}[1])
        & \simeq \Rg(\DD_\S, \mathbb{R}\pi_*((^\dag X_s)\Omega^\cdot_{\X/\DD_\S, \Q})) \notag\\
        & \simeq \Rg(\X, (^\dag X_s)\Omega^\cdot_{\X/\DD_\S, \Q}) \notag\\
        & \simeq \Rg(\X, \mathbb{R} {\rm sp}_* j^\dag_{X_\eta} \Omega^\cdot_{\X^{\rm ad}/\DD_\S^{\rm ad}}) \notag\\
        & \simeq \Rg(\X^{\rm ad}, j^\dag_{X_\eta} \Omega^\cdot_{\X^{\rm ad}/\DD_\S^{\rm ad}}) \notag\\
        & \simeq \Rg_{\rm rig}(X_\eta/\E^\dag_K). \notag
    \end{align}
    Taking cohomology, we get the isomorphism $\Rg_\D(X_\eta/\E^\dag_K) \simeq \Rg_{\rm rig}(X_\eta/\E^\dag_K)$. As in the case of algebraic $\D$-modules, the Gauss-Manin connection on $\mathbb{R}^{q}\pi_*((^\dag X_s) \Omega^\cdot_{\X/\DD_\S, \Q})$ is compatible with the structure of $\D^\dag$-module on $\pi^{q+1}_{+}(\one_{X_\eta}).$
    
    \begin{theorem} \label{main-theorem2}
        Let $X$ be a projective, strictly semi-stable scheme over $k[[t]]$. If $X$ lifts to a formal scheme over $\V[[t]]$ locally with finite $p$-basis over $\V$, then we have a spectral sequence
        $$ E^{p,q}_1 = \bigoplus_{i \geq \max(0, -p)} H^{q-2i}_{\rm rig}(D_{(p+2i)}/K)(-i) \Rightarrow \Psi H^{p+q}_{\rm rig}(X_\eta/\E^\dag_K).$$
    \end{theorem}
    
    \section{Functorality} \label{sec-func}
    Let $f : X \rightarrow X'$ be a morphism of projective, strictly semi-stable schemes over $k[[t]]$.
    Let $D_0, \dots D_r$ be the irreducible components of $X_s$ and $D'_0, \dots, D'_{r'}$ those of $X'$. As in \cite{l-adic}, we may assume that there is a nondecreasing map $\varphi : \{0, \dots, r\} \rightarrow \{0, \dots, r'\}$ such that $f(D_i) \subseteq D_{i'}$ if and only if $i' = \varphi(i).$

    We suppose that $f$ extends to a morphism of frames $\theta = (f,f,g) : (X, X, \P) \rightarrow (X', X', \P').$
    Let $j : (X_\eta, X, \P) \rightarrow (X, X, \P)$ and $j' : (X'_\eta, X', \P') \rightarrow (X', X', \P')$ be open immersions.
    We denote by $\Psi$ (resp. $\Psi'$) the (unipotent) nearby cycle of $\O_{X_\eta}$ (resp. $\O_{X'_\eta}$).
    
    \subsection{Pushforward}
    \begin{proposition} \label{hdag_finv}
        Suppose that there exists a lift $\X \hookrightarrow \P$ (resp. $\X' \hookrightarrow \P'$) of the closed immersion $X \hookrightarrow P$ (resp. $X' \hookrightarrow P$), where $\X$ (resp. $\X'$) locally has finite $p$-basis over $\V$. We also assume that $f$ lifts to a morphism $\X \rightarrow \X'$, which by abuse of notation we will denote by $f$.
        Let $I' \subseteq \{0, \dots, r'\}, \ \card I' = p'.$ Then $\H^{\dag i}_{f^{-1}(D'_{I'})}(\O_{\X, \Q}) = 0$ for $i \neq p'$ and
        $$f^! \H^{\dag p'}_{D'_{I'}}(\O_{\X', \Q}) \simeq \H^{\dag p'}_{f^{-1}(D'_{I'})}(\O_{\X, \Q})[\delta_X - \delta_{X'}].$$
    \end{proposition}
    \begin{proof}
        For $p \geq p',$ define ${\bold I}_{I', p} := \{I \subseteq \{0, \dots, r\} : \card I = p, \ \phi(I) = I'\}.$ We have $\displaystyle f^{-1}(D_{I'}) = \bigcup_{I \in {\bold I}_{I', p'}} D_I.$
        For a $\O_{\X^{\rm ad}}$-module $E$, define
        $$\displaystyle \mathcal{C}'^\dag_{f, I'}(E) := [\cdots \rightarrow \bigoplus_{I \in {\bold I}_{I', p+1} }\G_{D_I} E \rightarrow \bigoplus_{I \in {\bold I}_{I', p} }\G_{D_I} E \rightarrow \cdots \rightarrow \bigoplus_{I \in {\bold I}_{I', p'} }\G_{D_I} E \rightarrow E],$$
        where the right-most term is at degree zero.
        
        Let us prove that the canonical morphism $\mathcal{C}'^\dag_{f, I'}(E) \rightarrow j^\dag_{X - f^{-1}(D'_{I'})} E$ is a quasi-isomorphism by induction on $p'$. If $p' = 1$, this is Lemma \ref{GY-resol}.
        For a complex $K^\cdot$, we can define the double complex $\mathcal{C}'^\dag_{f, I'}(K^\cdot).$ If we take any $i' \in I'$ and put $I'' = I' - \{i'\}$, the total complex of $\mathcal{C}'^\dag_{f, \{i'\}}(\mathcal{C}'^\dag_{f, I''}(E))$ is just $\mathcal{C}'^\dag_{f, I'}(E).$ By inductive hypothesis, $\mathcal{C}'^\dag_{f, I''}(E) \xrightarrow{\sim} j^\dag_{X - f^{-1}(D'_{I''})} E$ is a quasi-isomorphism and since $j^\dag_{X - f^{-1}(D'_{i'})} \, j^\dag_{X - f^{-1}(D'_{I''})}\, E \simeq j^\dag_{X - f^{-1}(D'_{I'})}\, E,$ we get the assertion.
        As in \ref{RG-calc}, we can replace each $\G_{D_I} \O_{\X^{\rm ad}}$ by an ${\rm sp}_*$-acyclic resolution to get a double complex and a spectral sequence, which degenerates at page 2 and implies that $\H^{\dag i}_{f^{-1}(D'_{I'})}(\O_{\X, \Q}) = 0$ for $i \neq p'$.
    \end{proof}
    
    \begin{proposition} \label{f!-sheaves}
        Let $\alpha : j'_! I_X^{0, N} \rightarrow j'_+ I_X^{0, N}$ be the canonical morphism and put $K = \ker \alpha, \ C = \coker \alpha.$
        Then $f^! K, \ f^! j'_! I_{X'}^{0, N}, \ f^! j'_{!+} I_{X'}^{0, N}, \ f^! j'_+ I_{X'}^{0, N}, \ f^! C$ are $(\delta_X - \delta_{X'})$-shifted holonomic modules.
    \end{proposition}
    \begin{proof}
        The assertion is local on $\P$ and we may assume that there exists a lift $\X \hookrightarrow \P$ (resp. $\X' \hookrightarrow \P'$) of the closed immersion $X \hookrightarrow P$ (resp. $X' \hookrightarrow P$), where $\X$ (resp. $\X'$) locally has finite $p$-basis over $\V$. We also assume that $f$ lifts to $\X \rightarrow \X'$, which by abuse of notation will be denoted by $f$.
        
        (i) Let us prove that $f^!j'_+I_{X'}^{0, N}$ is a $(\delta_X - \delta_{X'})$-shifted holonomic module. Since $I_{X'}^{0, N}$ admits a filtration $0 \subseteq I_{X'}^{N-1, N} \subseteq \cdots \subseteq I_{X'}^{0,N}$ whose graded pieces are all isomorphic to $\O_{\X'}(^\dag X'_s)_\Q,$ it suffices to show that $f^! \O_{\X'}(^\dag X'_s)_\Q$ is a $(\delta_X - \delta_{X'})$-shifted holonomic module. By \cite[7.2.1]{intro}, we have $f^! \O_{\X'}(^\dag X'_s)_\Q \simeq \O_\X(^\dag X_s)_\Q[\delta_X - \delta_{X'}]$, which implies the assertion.

        (ii) By Proposition \ref{ie_form}, to show that $f^!j_{!+}I_X^{0, N}$ and $f^! C$ are $(\delta_X - \delta_{X'})$-shifted holonomic modules, it suffices to show that $f^! Gr^F_p \O_{\X'}(^\dag X'_s)_\Q$ are $(\delta_X - \delta_{X'})$-shifted holonomic modules. This follows from Theorem \ref{calc_rg} and Proposition \ref{hdag_finv}.

        (iii) For $f^! j'_! I_{X'}^{0, N}$, as there is an isomorphism $\DD I_{X'}^{0, N} \simeq I_{X'}^{0, N},$ it suffices to show that $f^! \DD j_+ I_{X'}^{0, N}$ is a $(\delta_X - \delta_{X'})$-shifted holonomic module. Again considering graded pieces, it suffices to show $f^! \DD Gr^F_p \O_{\X'}(^\dag X'_s)_\Q$ are $(\delta_X - \delta_{X'})$-shifted holonomic modules. By Theorem \ref{calc_rg}, $\DD Gr^F_p \O_{\X'}(^\dag X'_s)_\Q \simeq Gr^F_p \O_{\X'}(^\dag X'_s)_\Q$ and we get the assertion.

        Lastly, by the isomorphism $\DD_{X'} I_{X'}^{0, N} \simeq I_{X'}^{0, N},$ we have $\DD K \simeq C.$ Using arguments in (ii) and (iii), we get the assertion.
    \end{proof}
    
    By the proof, we also get the following proposition:
    
    \begin{proposition}
        $f^! \Psi', \ f^! F_p\Psi', \ f^! Gr^F_p\Psi', f^! G^q\Psi', \ f^! Gr_G^q\Psi',\  f^! F_pGr_G^q\Psi', \ f^!G^qGr^F_p\Psi', \ f^! Gr^F_pGr_G^q\Psi'$ are $(\delta_X - \delta_{X'})$-shifted holonomic modules.
    \end{proposition}
    
    We have a commutative diagram
    \[
    \xymatrix{
        X_\eta \ar[d]_{f_\eta} \ar[r]^{j} & X \ar[d]^f\\
        X'_\eta \ar[r]^{j'} & X'.\\
    }
    \]
    Since $f$ is proper, we have an isomorphism $f_! \xrightarrow{\sim} f_+$ and by adjunctions we can form $j_! f_\eta^! \rightarrow f^!j'_!$ and $j_+ f_\eta^! \rightarrow f^!j'_+$, which are compatible with $j_! \rightarrow j_+$. We therefore have a commutative diagram
    \begin{equation} \label{push-bc}
        \xymatrix{
           j_!f_\eta^! \ar[d] \ar[r] & f^!j'_! \ar[d]\\
           j_+f_\eta^! \ar[r] & f^!j'_+.
        }
    \end{equation}
    Suppose that we have lifts $\pi : \X \rightarrow \DD_\S$, $\pi' : \X' \rightarrow \DD_\S$ and $f : \X \rightarrow \X'$. In this case, by \cite[5.5]{frob} (which also holds in our setting) we have $\pi_\eta^! I_{\DD_S}^{0, N} \simeq \pi_\eta^+ I_{\DD_S}^{0, N}(\delta_X - 1)[2(\delta_X-1)] = I_X^{0,N}(\delta_X-1)[\delta_X-1]$ and ${\pi'}_\eta^! I_{\DD_S}^{0, N} \simeq I_{X'}^{0,N} (\delta_{X'}-1)[\delta_{X'}-1]$. Therefore $f_\eta^!I_{X'}^{0,N} \simeq f_\eta^!{\pi'}_\eta^! I_{\DD_S}^{0, N}(1-\delta_{X'})[1-\delta_{X'}] \simeq \pi_\eta^! I_{\DD_S}^{0, N}(1-\delta_{X'})[1-\delta_{X'}] \simeq I_X^{0,N}(\delta_X-\delta_X')[\delta_X - \delta_{X'}].$
    \\
    By Proposition \ref{f!-sheaves} and the commutative diagram (\ref{push-bc}), we get a canonical morphism
    \begin{align}
        \coker(j_!I_{X}^{0,N} \rightarrow j_+I_{X}^{0,N})(\delta_X-\delta_{X'})[\delta_X-\delta_{X'}] & \simeq \coker(j_!f_\eta^!I_{X'}^{0,N} \rightarrow j_+f_\eta^!I_{X'}^{0,N}) \notag\\
        & \rightarrow \coker(f^!j'_!I_{X'}^{0,N} \rightarrow f^!j_+I_{X'}^{0,N}) \notag\\
        & \simeq f^! \coker(j'_!I_{X'}^{0,N} \rightarrow j_+I_{X'}^{0,N}). \notag
    \end{align}
    Taking $N$ large enough, we get the morphism
    \begin{align} \label{push-nc}
        \Psi[\delta_X - \delta_{X'}] \rightarrow f^!\Psi'.
    \end{align}
    It is easy to see that this morphism is compatible with the multiplication by $x$. Thus we have a canonical morphism
    \begin{align} \label{push-grnc}
        Gr^F_pGr_G^q \Psi[\delta_X - \delta_{X'}] \rightarrow Gr^F_pGr_G^q f^! \Psi' \simeq f^! Gr^F_pGr_G^q \Psi'.
    \end{align}
    Let us investigate the relation of the morphism (\ref{push-grnc}) and the isomorphism in Theorem \ref{cons_trun}. Recall that $I_X^{0,N}$ is represented by $\displaystyle \bigoplus_{0 \leq k < N} \O_\X(^\dag X_s)_\Q \, x^k t^x$.
    Since $f^!j'_+ I_{X'}^{0,N} = \D^\dag_{\X \rightarrow \X', \Q} \otimes^{\mathbb{L}}_{f^{-1}\D^\dag_{\X', \Q}} f^{-1} I_{X'}^{0,N}[\delta_X - \delta_{X'}]$ is a $(\delta_X - \delta_{X'})$-shifted sheaf, the canonical isomomorphism $\rho : f^! j'_+ I_{X'}^{0,N}[\delta_{X'} - \delta_X] \xrightarrow{\sim} j_+ f_\eta^!I_{X'}^{0,N}[\delta_{X'} - \delta_X] \xrightarrow{\sim} j_+ I_X^{0,N}$ is given by
    $$\D^\dag_{\X \rightarrow \X', \Q} \otimes_{f^{-1}\D^\dag_{\X', \Q}} f^{-1} I_{X'}^{0,N} \rightarrow I_{X}^{0,N}, \ (a \hat{\otimes} Q)\otimes\xi \mapsto af^\#(Q(\xi)),$$  
    where $a \in \O_{\X, \Q}, \ Q \in f^{-1}\D^\dag_{\X', \Q}, \ \xi \in f^{-1} I_{\X'}^{0,N}$ and $f^\# : f^{-1} I_{X'}^{0,N} \rightarrow I_{X}^{0,N}$ is induced by the canonical homomorphism $f^\# : f^{-1}\O_{\X'}(^\dag X'_s)_{\Q} \rightarrow \O_{\X}(^\dag X_s)_{\Q}$.

    By Proposition \ref{ie_form}, the image of $f^!j'_{!+} I_{X'}^{0, N}[\delta_{X'} - \delta_X]$ by $\rho$ is $\displaystyle \bigoplus_{0 \leq k < N} F'_k \O_\X(^\dag X_s)_\Q \, x^k t^x,$ where $F'_k$ is the $\D^\dag_{\X, \Q}$-submodule of generated by $f^\#(f^{-1} F_k \O_{\X'}(^\dag X'_s)).$ Explicitly, $\displaystyle F'_k = \sum_{i_1 < \cdots < i_k} \O_\X(^\dag f^{-1}(D'_{i_1} \cup \cdots \cup D'_{i_k}))_\Q.$ 
    Thus we get $\displaystyle f^! \Psi'[\delta_{X'} - \delta_X] = \bigoplus_{0 \leq k < N} \O_{\X}(^\dag X_s)_\Q / F'_k\O_{\X}(^\dag X_s)_\Q \, x^k t^x.$
    Since $D_i \subseteq f^{-1}(D'_{\varphi(i)}),$ we have $F_k \O_{\X}(^\dag X_s)_\Q \subseteq F'_k \O_{\X}(^\dag X_s)_\Q$. The homomorphism (\ref{push-nc}) is the canonical homomorphism
    $$\displaystyle \bigoplus_{0 \leq k < N} \O_{\X}(^\dag X_s)_\Q / F_k\O_{\X}(^\dag X_s)_\Q \, x^k t^x \rightarrow \bigoplus_{0 \leq k < N} \O_{\X}(^\dag X_s)_\Q / F'_k\O_{\X}(^\dag X_s)_\Q \, x^k t^x$$
    with obvious Frobenius structures, and therefore (\ref{push-grnc}) is the canonical homomorphism
    \begin{equation} \label{push-nc-gr1}
        Gr^F_k \O_\X(^\dag X_s)_\Q(k) \rightarrow Gr^{F'}_k \O_\X(^\dag X_s)_\Q(k).
    \end{equation}
    with $k = p+q+1$, induced by $F_k \O_{\X}(^\dag X_s)_\Q \subseteq F'_k \O_{\X}(^\dag X_s)_\Q$. Since the filtration $F'_k \O_\X(^\dag X_s)_\Q$ is isomorphic to $f^! F_k \O_{\X'}(^\dag X'_s)_\Q[\delta_{X'} - \delta_X],$ we have 
    \begin{equation} \label{push-nc-gr2}
        Gr^{F'}_k \O_\X(^\dag X_s)_\Q \simeq f^! Gr^F_k \O_{\X'}(^\dag X'_s)_\Q[\delta_{X'} - \delta_X].
    \end{equation}
    Fix an $I = \{i_1 < \cdots < i_k\}$. As in the proof of Proposition \ref{alphabeta}, we have a canonical morphism
    $$
        \iota_I : \mathcal{C}^\dag_{\X, I} = [\O_{\X, \Q} \rightarrow \bigoplus_{j_1} \O_\X(^\dag D_{i_{j_1}})_\Q \rightarrow \cdots \rightarrow \O_\X(^\dag D_{i_1} \cup \cdots \cup D_{i_k})_\Q] \rightarrow Gr^F_k \O_{\X}(^\dag X_s)_\Q.
    $$
    The direct sum $\bigoplus \iota_I$ gives the isomorphism of Remark \ref{cons_trun_form}.
    Put $i'_j := \varphi(i_j) $ and $I' = \{i'_1, \dots, i'_k\}.$ If $\card I' < k$, then the composition $\mathcal{C}^\dag_{\X, I} \xrightarrow{\iota_{I}} Gr^F_k \O_{\X}(^\dag X_s)_\Q \rightarrow Gr^{F'}_k \O_{\X}(^\dag X_s)_\Q$ is zero. Suppose $\card I' = k.$ Put $U_j = X - f^{-1}(D'_{i'_j})$ and $U_{j_1 \cdots j_m} = U_{j_1} \cap \cdots \cap U_{j_m}.$ We have an ${\rm sp}_*$-acyclic resolution
    $$\displaystyle \G_{f^{-1}(D'_{I'})} \O_{\X^{\rm ad}} \xrightarrow{\sim} [\O_{\X^{\rm ad}} \rightarrow \bigoplus_{j_1} j^\dag_{U_{j_1}}\O_{\X^{\rm ad}} \rightarrow \bigoplus_{j_1 < j_2} j^\dag_{U_{j_1 j_2}} \O_{\X^{\rm ad}} \rightarrow \cdots \rightarrow j^\dag_{U_{1\cdots n}} \O_{\X^{\rm ad}}].$$
    Thus $\displaystyle \RG_{f^{-1}(D'_{I'})}(\O_{\X, \Q}) \simeq [\O_{\X, \Q} \rightarrow \bigoplus_{j_1} \O_\X(^\dag f^{-1}(D'_{i'_{j_1}}))_\Q \rightarrow \cdots \rightarrow \O_\X(^\dag f^{-1}(D'_{i_1} \cup \cdots \cup D'_{i_k}))_\Q] =: \mathcal{C}'^\dag_{\X, I'}.$ We have a commutative diagram
    \begin{equation} \label{push_gr1}
    \xymatrixcolsep{20mm}
    \xymatrix{
        \mathcal{C}^\dag_{\X, I} \ar[r]^(0.40){\iota_I}\ar[d] & Gr^F_k \O_{\X}(^\dag X_s)_\Q \ar[d] \\
        \mathcal{C}'^\dag_{\X, I'} \ar[r]^(0.40){\iota'_{I'}} & Gr^{F'}_k \O_{\X}(^\dag X_s)_\Q,
    }
    \end{equation}
    where $\iota'_{I'}$ is defined as $\iota_I$ and vertical arrows are canonical maps.

    Since $f^! \O_{\X'}(^\dag D'_{i'_{j_1}} \cup \cdots \cup D'_{i'_{j_l}})_\Q \simeq \O_\X(^\dag f^{-1}(D'_{i'_{j_1}} \cup \cdots \cup D'_{i'_{j_l}}))_\Q[\delta_X-\delta_{X'}]$, we have
    \begin{equation} \label{push_gr2}
        f^! \mathcal{C}^\dag_{\X', I'} \simeq \mathcal{C}'^\dag_{\X, I'}[\delta_X-\delta_{X'}].
    \end{equation}
    This induces the canonical commutative diagram
    \begin{equation} \label{push_gr_diag}
    \xymatrixcolsep{25mm}
    \xymatrix{
        \mathcal{C}^\dag_{\X, I} \ar[r]^(0.43){\iota_I}\ar[d] & Gr^F_k \O_{\X}(^\dag X_s)_\Q \ar[d] \\
        f^! \mathcal{C}^\dag_{\X', I'}[\delta_{X'}-\delta_X] \ar[r]^(0.43){f^!\iota_{I'}[\delta_{X'}-\delta_X]} & f^! Gr^{F}_k \O_{\X'}(^\dag X'_s)_\Q[\delta_{X'}-\delta_X],
    }
    \end{equation}
    The isomorphism (\ref{push-nc-gr2}) is compatible with (\ref{push_gr2}) the following proposition:
    \begin{proposition}
        Suppose that $f : X \rightarrow X'$ lifts to $f : \X \rightarrow \X'$. Then we have a commutative diagram
        \[\hspace{-30pt}
        \xymatrix{
            Gr^F_pGr_G^q \Psi \ar[r]{(\ref{push-grnc})} \ar[d]^{\wr}_{\ref{nc_gr}} & f^! Gr^F_pGr_G^q \Psi' [\delta_{X'}-\delta_X] \ar[d]^{\ref{nc_gr}}_{\wr} \\
            Gr^F_{p+q+1} \O_{\X}(^\dag X_s)_\Q(q+1) \ar@{-->}[r] & f^! Gr^F_{p+q+1} \O_{\X'}(^\dag X'_s)_\Q(q+1)[\delta_{X'}-\delta_X]\\
        \displaystyle \bigoplus_{\card I' = p+q+1} \RG_{D_I} \O_{\X}(^\dag X_s)_\Q(q+1)[\delta_X - (p+q+1)] \ar[u]^{\ref{cons_trun}}_{\wr} \ar[r] & \displaystyle\bigoplus_{\card I' = p+q+1} \RG_{f^{-1}(D'_{I'})} \O_{\X}(^\dag X_s)_\Q(q+1)[\delta_X - (p+q+1)] \ar[u]^{\wr}_{\ref{cons_trun}},
        }
        \]
        where the right bottom vertical isomorphism comes from
       $$f^! \RG_{D'_{I'}} \O_{\X', \Q} \simeq \RG_{f^{-1}(D'_{I'})} f^!\O_{\X',_\Q} \simeq \RG_{f^{-1}(D'_{I'})} \O_{\X,_\Q}[\delta_X - \delta_{X'}]$$
       and the bottom horizontal morphism is defined by the canonical morphisms $\RG_{D_I} \rightarrow \RG_{f^{-1}(D'_{I'})}$ for $\card I = \card I' = p+q+1$ such that $D_I \subseteq f^{-1}(D'_{I'}).$
    \end{proposition}
    \begin{proof}
        The middle dashed arrow comes from (\ref{push-nc-gr1}), (\ref{push-nc-gr2}) and so the top square is commutative. Since the isomorphism (\ref{push_gr2}) is compatible with $\iota$ and $f^! \iota$, by the commutativity of (\ref{push_gr1}), the boottom square is commutative.
    \end{proof}
    
    Noting that this commutative diagram can be glued via Theorem \ref{u_0^*}. By rewriting it using $R\psi\one$, we get the following theorem which is analogous to the $l$-adic case \cite[2.13]{l-adic}:
    
    \begin{theorem} \label{pushforward-psi}
        Let $d$ = $\delta_X - \delta_{X'}$. Then we have a commutative diagram
        \[
        \xymatrix{
            Gr^M_n R\psi \one_{X_\eta} (d)[2d] \ar[r] \ar[d]^{\wr} & f^! Gr^M_n R\psi \one_{X'_\eta} \ar[d]^{\wr}\\
            \displaystyle \bigoplus_{\substack{p-q = n \\ \card I = p+q+1}} \one_{D_I}(d-p)[2d-(p+q)] \ar[r] & \displaystyle \bigoplus_{\substack{p-q = n \\ \card I' = p+q+1}} f^!\one_{D'_{I'}}(-p)[-(p+q)],
        }
        \]
        where the bottom arrow is the canonical morphism.
    \end{theorem}
    
    \begin{corollary}
        Let $d = \delta_X - \delta_{X'}$. We have a map of spectral seqences
        \[
        \xymatrix{
            E_1^{p, q+2d} = \displaystyle \bigoplus_{i \geq \max(0,-p)} H_{\D}^{q+2d-2i}(D_{(p+2i)}/K)(-i+d) \ar@{=>}[r] \ar[d]_(0.55){\bigoplus f_{(p+2i)+}} & \Psi H_{\D}^{p+q+2d}(X_\eta/\E^\dag_K)(d) \ar[d]^{\Psi f_+} \\
            E_1'^{p, q} = \displaystyle \bigoplus_{i \geq \max(0,-p)} H_{\D}^{q-2i}(D'_{(p+2i)}/K)(-i) \ar@{=>}[r] & \Psi H_{\D}^{p+q}(X'_\eta/\E^\dag_K).
        }
        \]
    \end{corollary}
    
    \begin{corollary}
        Suppose that Conjecture \ref{conj-comp} is true. Let $d$ = $\delta_X - \delta_{X'}$. We have a map of spectral sequences
        \[
        \xymatrix{
            E_1^{p, q+2d} = \displaystyle \bigoplus_{i \geq \max(0,-p)} H_{\rm rig}^{q+2d-2i}(D_{(p+2i)}/K)(-i+d) \ar@{=>}[r] \ar[d] & \Psi H_{\rm rig}^{p+q+2d}(X_\eta/\E^\dag_K)(d) \ar[d] \\
            E_1'^{p, q} = \displaystyle \bigoplus_{i \geq \max(0,-p)} H_{\rm rig}^{q-2i}(D'_{(p+2i)}/K)(-i) \ar@{=>}[r] & \Psi H_{\rm rig}^{p+q}(X'_\eta/\E^\dag_K),
        }
        \]
        where vertical arrows are the dual of Poincar\'{e} duality.
    \end{corollary}
    \subsection{Pullback}
    \begin{proposition}
        With the notation of Proposition \ref{f!-sheaves}, $f^+ K, \ f^+ j'_+ I_{X'}^{0, N}, \ f^+ j'_{++} I_{X'}^{0, N}, \ f^+ j'_+ I_{X'}^{0, N}, \allowbreak f^+ C$ are $(\delta_{X'} - \delta_X)$-shifted holonomic modules.
        Furthermore, $f^+ \Psi', \ f^+ F_p\Psi', \ f^+ Gr^F_p\Psi', f^+ G^q\Psi', \allowbreak f^+ Gr_G^q\Psi', f^+ F_pGr_G^q\Psi', \ f^+G^qGr^F_p\Psi', \ f^+ Gr^F_pGr_G^q\Psi'$ are $(\delta_{X'} - \delta_{X})$-shifted holonomic modules.
    \end{proposition}
    \begin{proof}
        Using the isomorphism $\DD(Gr_k \O_\X(^\dag X_s)_\Q) \simeq Gr_k \O_\X(^\dag X_s)_\Q$, the proof is similar to that in Proposition \ref{f!-sheaves}.
    \end{proof}
    
    We have a commutative diagram (which is dual to (\ref{push-bc}))
    \begin{equation} \label{pull-bc}
        \xymatrix{
           f^+j'_! \ar[d] \ar[r] & j_!f_\eta^+ \ar[d]\\
           f^+j'_+ \ar[r] & j_+f_\eta^+.
        }
    \end{equation}
    This induces a morphism of $(\delta_{X'} - \delta_{X})$-shifted holonomic modules
    \begin{equation} \label{pull-nc}
        f^+\Psi'(-\delta_{X'}) \rightarrow \Psi(-\delta_X)[\delta_{X'} - \delta_X]. \notag
    \end{equation}
    Rewriting this using $R\psi\one$, we have a morphism of $(1-\delta_{X})$-shifted holonomic modules
    \begin{equation} \label{pull-nc}
        f^+R\psi\one_{X'_\eta} \rightarrow R\psi\one_{X_\eta}.
    \end{equation}
    This induces a morphism on graded pieces:
    \begin{equation} \label{pull-grnc}
        f^+Gr^F_p Gr_G^q R\psi\one_{X'_\eta} = Gr^F_p Gr_G^q f^+R\psi\one_{X'_\eta} \rightarrow Gr^F_p Gr_G^qR\psi\one_{X_\eta}.
    \end{equation}
    
    On the other hand, fix $I' \subseteq \{0, \dots, r'\}$ and $I \subseteq \{0, \dots, r\}$ such that $\card I = \card I' = k$ and $f(D_I) \subseteq D'_{I'}.$ We have a commutative diagram
    \[
        \xymatrix{
            D_I \ar@{^{(}->}[r]^{a_I} \ar[d]_{f_{II'}} & X \ar[d]^f\\
            D'_{I'} \ar@{^{(}->}[r]^{a'_{I'}} & X'.
        }
    \]
    Then we have a canonical morphism $f^+ a_{I'+} \one_{D'_{I'}} \xrightarrow{f^+} a_{I+}f_{II'}^+ \one_{D'_{I'}} \simeq a_{I+} \one_{D_I}.$ By taking the direct sum, we get a morphism
    \begin{equation} \label{pull_one}
        \displaystyle \bigoplus_{\card I' = n} f^+\one_{D'_{I'}} \xrightarrow{\oplus f^+} \bigoplus_{\card I = n} \one_{D_{I}}.
    \end{equation}
    We state the conjecture which is an analogue of $l$-adic case \cite[2.11]{l-adic}:
    \begin{conjecture} \label{pullback-psi}
        The following diagram is commutative:
        \[
         \xymatrix{
             f^+Gr^F_p Gr_G^q R\psi\one_{X'_\eta} \ar[r]^{(\ref{pull-grnc})}\ar[d]^{\wr}_{(\ref{rpsi_grgr})} & Gr^F_p Gr_G^qR\psi\one_{X_\eta} \ar[d]^{(\ref{rpsi_grgr})}_{\wr}\\
             \displaystyle \bigoplus_{\card I' = p+q+1} f^+\one_{D'_{I'}}(-p)[-(p+q)] \ar[r]^{(\ref{pull_one})} & \displaystyle \bigoplus_{\card I = p+q+1} \one_{D_I}(-p)[-(p+q)]
        }
        \]
    \end{conjecture}
    This conjecture induces straightforward corollaries:
    \begin{corollary} \label{pullback-wss}
        Under the Conjecture \ref{pullback-psi}, we have a morphism of spectral sequeneces
        \[
        \xymatrix{
            E_1'^{p, q} = \displaystyle \bigoplus_{i \geq \max(0,-p)} H_{\rm rig}^{q-2i}(D'_{(p+2i)}/K)(-i) \ar@{=>}[r] \ar[d]_(0.55){\bigoplus f^+_{(p+2i)}} & \Psi H_{\rm rig}^{p+q}(X'_\eta/\E^\dag_K) \ar[d]^{\Psi f^+}\\
            E_1^{p, q} = \displaystyle \bigoplus_{i \geq \max(0,-p)} H_{\rm rig}^{q-2i}(D_{(p+2i)}/K)(-i) \ar@{=>}[r] & \Psi H_{\D}^{p+q}(X_\eta/\E^\dag_K).
        }
        \]
    \end{corollary}
    
    \begin{corollary}
        Under the Conjecture \ref{conj-comp} and \ref{pullback-psi}, we have a morphism of spectral sequeneces
        \[
        \xymatrix{
            E_1'^{p, q} = \displaystyle \bigoplus_{i \geq \max(0,-p)} H_{\rm rig}^{q-2i}(D'_{(p+2i)}/K)(-i) \ar@{=>}[r] \ar[d]_(0.55){\bigoplus f^*_{(p+2i)}} & \Psi H_{\rm rig}^{p+q}(X'_\eta/\E^\dag_K) \ar[d]^{\Psi f^*} \\
            E_1^{p, q} = \displaystyle \bigoplus_{i \geq \max(0,-p)} H_{\rm rig}^{q-2i}(D_{(p+2i)}/K)(-i) \ar@{=>}[r] & \Psi H_{\rm rig}^{p+q}(X_\eta/\E^\dag_K).
        }
        \]
    \end{corollary}
    \subsection{Dual} \label{func_dual}
    Let $F_p \DD(\Psi) := \DD(\Psi/G^{p+1}\Psi), \ G^q \DD(\Psi) := \DD(\Psi/F_{q-1}\Psi)$ and $M_r \DD(\Psi) := \DD(\Psi/M_{-r-1}\Psi).$ Then by the characterizations of kernel (resp. image, monodromy) filtrations, $F_p \DD(\Psi)$ (resp. $G^q \DD(\Psi), \, M_r \DD(\Psi)$) is the kernel (resp. image, monodromy) filtration of $\DD(\Psi)$ with respect to the nilpotent endomorphism $\DD(x).$ Therefore we have $Gr^F_pGr_G^q \DD(\Psi) \simeq \DD(Gr_G^pGr^F_q \Psi).$
    
    By the commutative diagram of Theorem \ref{nc_form}, $b$ induces an isomorphism $Gr^F_pGr_G^q \DD(\Psi)(1) \xrightarrow{Gr^F_pGr_G^q b} Gr^F_pGr_G^q \Psi(-\delta_X).$
    
    \begin{conjecture}[{cf. \cite[2.15]{l-adic}}] \label{dual-nc}
        We have a commutative diagram
        \[\hspace{-30pt}
        \xymatrix{
            \DD(Gr^F_pGr_G^q \Psi) \ar[r]^{\sim} \ar[d]^{\ref{nc_gr}}_{\wr} & Gr_G^pGr^F_q \DD(\Psi) \ar[r]^{Gr^F_pGr_G^q b}_{\sim} & Gr_G^p Gr^F_q \Psi(-1-\delta_X) \ar[d]^{\ref{nc_gr}}_{\wr}\\
            \displaystyle \bigoplus_{\card I = p+q+1} \DD(\RG_{D_I}(\O_{\X, \Q})(q+1)[p+q+1]) \ar[rr]^{\sim} & & \displaystyle \bigoplus_{\card I = p+q+1} \RG_{D_I}(\O_{\X, \Q})(p-\delta_X)[p+q+1],
        }
        \]
        where the bottom isomorphism is induced by $\DD(\RG_{D_I}(\O_{\X, \Q})[p+q+1]) \simeq \RG_{D_I}(\O_{\X, \Q})(-\delta_X+p+q+1)[p+q+1]$ (see {\rm \cite[9.4.11]{intro}}).
    \end{conjecture}
    
    \begin{theorem} \label{dual-and-pull}
    Conjecture \ref{dual-nc} together with Theorem \ref{pushforward-psi} implies Conjecture \ref{pullback-psi} and consequently Corollary \ref{pullback-wss}.
    \end{theorem}
    \begin{proof}
    Let us first observe that we have a commutative diagram
    \begin{equation} \label{pull-dual}
    \xymatrix{
        f^+ \DD \Psi' \ar[r]^{\DD f_+}\ar[d]^{\wr}_{f^+b'} & \DD \Psi[-d] \ar[d]^{b}_{\wr} \\
        f^+ \Psi'(-1-\delta_{X'}) \ar[r]^{f^+} & \Psi(-1-\delta_{X}) [-d],
    }
    \end{equation}
    where $d = \delta_X - \delta_{X'}$. This comes from the functoriality of connecting homomorphisms $b$ and $b'$. Thus we have a commutative diagram
    \begin{equation} \label{pull-dual-grgr}
        \xymatrix{
            f^+\DD Gr_G^pGr^F_q \Psi' \ar[r]^{\DD Gr_G^pGr^F_q f_+} \ar[d]^\wr & \DD Gr_G^pGr^F_q \Psi [d] \ar[d]^\wr \\
            f^+ Gr^F_pGr_G^q\DD \Psi' \ar[d]^{\wr}_{f^+Gr^F_pGr_G^qb'} & Gr^F_pGr_G^q\DD \Psi[-d] \ar[d]^{Gr^F_pGr_G^q b}_{\wr} \\
            f^+ Gr^F_pGr_G^q\Psi'(-1-\delta_{X'}) \ar[r]^{Gr^F_pGr_G^qf^+} & Gr^F_pGr_G^q\Psi(-1-\delta_{X}) [-d].
        }
    \end{equation}
    Note that $\DD b' = b'$ and $\DD b = b$. This comes from the construction of $b$; since $b$ is the connecting homomorphism of a long exact sequence, it is functorial. By taking the dual of (\ref{pull-dual-grgr}), we get the following commutative diagram:
    \begin{equation} \label{push-grgr}
        \xymatrixcolsep{25mm}
        \xymatrix{
            \DD Gr^F_pGr_G^q \Psi[d] \ar[r]^{Gr_G^pGr^F_q \DD f^+}\ar[d]^\wr & f^! \DD Gr^F_p Gr_G^q \Psi' \ar[d]^\wr\\
            Gr_G^pGr^F_q \DD\Psi[d] \ar[d]^{\wr}_{Gr_G^pGr^F_q b} & f^! Gr_G^pGr^F_q \DD\Psi' \ar[d]^{Gr_G^pGr^F_q f^!b'}_{\wr}\\
            Gr_G^pGr^F_q \Psi(-1-\delta_X)[d] \ar[r]^{Gr_G^pGr^F_q f_+} & f^!Gr_G^pGr^F_q \Psi'(-1-\delta_{X'}).
        }
    \end{equation}
    Thus we can form the following diagram:
    \footnotesize
    \[\hspace{-35pt}
    \xymatrixcolsep{-12mm}
    \xymatrix{
        & \DD Gr^F_pGr_G^q\Psi[d] \ar[rr]^{Gr_G^pGr^F_q \DD f^+}\ar[ld]^\sim \ar@{-->}[dd]^(0.55){\wr} & & f^!\DD Gr^F_pGr_G^q \Psi' \ar[ld]^{\sim}\ar[dd]^\wr\\
        Gr_G^pGr^F_q \Psi(-1-\delta_X)[d] \ar[rr]^(0.60){Gr_G^pGr^F_q f_+}\ar[dd]^{\wr} & & f^!Gr_G^pGr^F_q \Psi'(-1-\delta_{X'})\ar[dd]_(0.55){\wr} & \\
        & \DD(\RG_{D^{(p+q)}}(\O_{\X, \Q})(q+1)[p+q+1])[d] \ar@{-->}[ld]^\sim \ar@{-->}[rr]^(0.45){\DD (\oplus f^+)} & & \DD(\RG_{D'^{(p+q)}}(\O_{\X', \Q})(q+1)[p+q+1]) \ar[ld]^\sim\\
        \RG_{D^{(p+q)}}(\O_{\X, \Q})(p-\delta_X)[p+q+1+d] \ar[rr]^{\oplus f_+} & & f^!\RG_{D'^{(p+q)}}(\O_{\X', \Q})(p-\delta_{X'})[p+q+1]. & \\
    }
    \]
    \normalsize
    Here $\RG_{D^{(k)}}$ is the abbreviation of $\displaystyle\bigoplus_{\card I = k+1} \RG_{D_I}$. The top square is (\ref{push-grgr}) and is commutative. The left and right squares are commutative by Conjecture \ref{dual-nc}. The commutativity of the bottom square comes from the functoriality of relative duality isomorphism $i_+ \DD \simeq \DD i_+$ for any closed immersion $i$. The square in front is commutative by Theorem \ref{pushforward-psi}. Therefore it follows that the square in the back, which is the dual of the diagram of Conjecture \ref{pullback-psi}, is commutative.
    \end{proof}
    
    \bibliographystyle{plain}
    \bibliography{weight_spectral_sequence_in_p-adic_cohomology}
\end{document}